\documentclass[final,3p,times]{elsarticle}
\usepackage{amsfonts}
\usepackage{dsfont}
\usepackage{latexsym,amssymb}
\usepackage{amsmath, amsbsy}
\usepackage{amsopn, amstext}
\usepackage{graphicx, xcolor, epstopdf}
\usepackage{caption}
\usepackage{float}
\usepackage{enumitem}
\usepackage{graphicx}
\usepackage{subfig}
\usepackage{diagbox}
\usepackage{amsmath,bm}
\usepackage{booktabs}
\usepackage{multirow}
\allowdisplaybreaks[4]

\usepackage[marginal]{footmisc}
\newtheorem{lemma}{Lemma}
\newtheorem{theorem}{Theorem}
\newtheorem{remark}{Remark}

    \numberwithin{Fig}{section}

\makeatletter 
\renewcommand{\section}{\@startsection{section}{1}{0mm}
  {-0.5\baselineskip}{0.3\baselineskip}{\bf\leftline}}

\let\oldthebibliography\thebibliography
\renewcommand\thebibliography[1]{
  \oldthebibliography{#1}
  \setlength{\itemsep}{0em} 
  \setlength{\parskip}{0pt} 
}
\numberwithin{equation}{section}
\numberwithin{Lem}{section} 
\numberwithin{Rem}{section} 
\numberwithin{theorem}{section}

\begin{document}
\begin{frontmatter}

\title{A virtual element approximation for the modified transmission eigenvalues for natural materials}

\author[1]{Liangkun Xu 
}\ead{xlkun@gznu.edu.cn}

\author[1]{Shixi Wang 
}\ead{wangshixi@gznu.edu.cn}

\author[1]{Hai Bi \corref{cor1}
}\ead{ bihaimath@gznu.edu.cn}

\cortext[cor1]{Corresponding Author}

\begin{abstract}
In this paper, we discuss a virtual element approximation for the modified transmission eigenvalue problem in inverse scattering for natural materials. In this case, due to the positive artificial diffusivity parameter in the considered problem, the sesquilinear form at the left end of the variational form is not coercive.
We first demonstrate the well-posedness of the discrete source problem using the $\mathds{T}$-coercivity property, then provide the a priori error estimates for the approximate eigenspaces and eigenvalues, and finally report several numerical examples. The numerical experiments show that the proposed method is effective.
\end{abstract}

\begin{keyword}
~the modified transmission eigenvalues; virtual element method; polygonal meshes; error estimation, the $\mathds{T}$-coercivity property.
\MSC  65N30\sep 65N25

\end{keyword}
\end{frontmatter}

\section{Introduction}\label{int}
\indent Transmission eigenvalues not only have important physical applications but also have theoretical importance in the uniqueness and reconstruction in inverse scattering theory. For example, transmission eigenvalues are often used for non-destructive testing and quantitative analysis of materials (see \cite{CakoniNHM2021,CakoniCandH2022,MengMei2022,Liu2023}, etc).
However, the method of using transmission eigenvalues as a target signature only applies to non-absorbing media (real refractive index) and relies on multifrequency scattering data. These restrictions hinder its applicability to absorbing media and lead to high data acquisition costs. To overcome these difficulties and provide effective target signatures for nondestructive material testing, the modified transmission eigenvalue problem was introduced in 2017~\cite{CoM2017,AudibertL2017} and has attracted increasing attention in recent years.
For instance,
\cite{Cogar2018} and \cite{CogarMonk2020} respectively studied the issues of modified transmission eigenvalues in partially coated crack scattering and the modified electromagnetic transmission eigenvalues in inverse scattering theory, in \cite{MonkSelgas2022,WangBY2023} the authors discussed the modified transmission eigenvalues for inverse scattering in a fluid-solid interaction problem, \cite{Zhang2023H} studied the multigrid method based on conforming FE for the modified elastic transmission eigenvalue problem, 
 \cite{Gintides2021} investigated a spectral Galerkin method for the modified transmission eigenvalue problem,  \cite{Li2023} explored a conforming finite element method (FEM), \cite{WangBiY2023} studied a mixed DG FEM, \cite{Meng2023} examined a virtual element method (VEM), 
 and \cite{FengLi2025} study the a posteriori error estimates of the conforming mixed method, etc.\\
\indent The VEM is an extension of the standard FEM by using general polygonal meshes for discretization. It was introduced in \cite{Beirao2013} and \cite{BeiraoBM2014}. The VEM offers greater flexibility compared to the standard FEM when dealing with partial differential equations on complex geometric domains or the ones associated with high-regularity admissible spaces. In recent years, VEM has been successfully applied to a variety of equations and eigenvalue problems, including the Stokes equations \cite{Cangiani2016,Zhao2019}, Navier-Stokes equations \cite{Gatica2018,Liu2019}, nonlinear Schr\"{o}dinger equations \cite{Lizhao2021}, nonlinear time-dependent convection-diffusion-reaction equations \cite{Arrutselvi2021}, the biharmonic equations \cite{Antonietti2018,Zhao2023}, Laplacian eigenvalue problems \cite{Meng2022,Meng2024}, Steklov eigenvalue problems \cite{Mora2015,Wang2022}, transmission eigenvalue problems \cite{MengMei2022,Mora2018,Meng2026XuS}, Stokes eigenvalue problems \cite{Lepe2021,Adak2024}, elliptic eigenvalue problems \cite{Gardini2018,Certik2020}, biharmonic eigenvalue problem \cite{MengXuS2025}, and Oseen eigenvalue problem \cite{AdakLeR2026}, etc.
For the modified transmission eigenvalue problem, \cite{Meng2023} studied its VE approximation in the context of artificial meta-materials. However, to the best of our knowledge, there have been no literature reports on the VE approximation in the case of natural materials.\\
\indent In this paper, based on the above work, we will explore a VE approximation for the modified transmission eigenvalue problem for natural materials. In this case, the artificial diffusivity parameter is positive, which prevents us from using the shift argument method in \cite{Meng2023} to prove the coercivity of the sesquilinear form in the variational formulation. We utilize the $\mathds{T}$-coercivity approach in \cite{CoM2017,DhiaAS2011} to discuss the well-posedness of the problem, and give the a priori estimates of the source problem.
Then, using the spectral approximation theory, we establish a complete error estimation of the eigenvalue problem with the help of the estimates of the source problem. We prove that the error estimate of approximate eigenfunctions in the $L^2$-norm is a higher-order quantity than the $H^1$-norm estimate. \\
\indent The remainder of this paper is organized as follows. The next section presents the variational formulation of the modified transmission eigenvalue problem. The third section gives the VE approximation for the modified transmission eigenvalue problem and provides the a priori error estimates.
The fourth section reports numerical examples on polygonal meshes to verify the efficiency and the accuracy of the proposed method.\\
\indent  In this paper, vector variables are represented by bold letters. Let $H^t(\Omega)$ denote the usual Sobolev space
over the domain $\Omega$ equipped with the norm $\|\cdot\|_{t,\Omega}$, and let $H^0(\Omega)=L^2(\Omega)$ with the inner product
$(\varphi,\psi)_{0,\Omega}=\int_\Omega \varphi \overline{\psi}$. The subscript $\Omega$ will be omitted when the computation domain is $\Omega$. Throughout the paper, we use the letter $C$, with or without subscripts, to denote a general constant that is independent of the mesh size $h$ and may take different values in different contexts. The notation “$a \lesssim b$” indicates that $a \leq Cb$, and “$a \gtrsim b$” indicates that $a \geq Cb$.

\section{The variational form of the modified transmission eigenvalue problem}\label{Sec:Dimen}
\indent Let $\Omega \in \mathbb{R}^2$ be a bounded region with Lipschitz boundary $\partial\Omega$. The modified transmission eigenvalue problem is to find $\lambda \in \mathds{C}$ and non-zero functions $w$ and $u$ such that
\begin{align}
\Delta w + k^2 n w & =  0 \quad  in ~\Omega,                   \tag{1a} \label{2.1a} \\
 \frac{1}{\gamma} \Delta u + \lambda k^2 u &= 0  \quad  in ~\Omega, \tag{1b}  \label{2.1b}\\
 w - u &=0   \quad  on ~\partial \Omega, \tag{1c}  \label{2.1c}\\
  \frac{\partial w}{\partial {\pmb \nu}} - \frac{1}{\gamma} \frac{\partial u}{\partial {\pmb \nu}}&=  0 \quad  on ~ \partial\Omega, \tag{1d} \label{2.1d}
\end{align}
\noindent where $n=n_1({\boldsymbol x})+n_2({\boldsymbol x})\mathrm{i} \in L^\infty(\Omega)$ represent the refractive index, $n_1({\boldsymbol x})>0$, and $n_2({\boldsymbol x})\geq 0$. The constant $k$ denotes a fixed positive wave number,
${1}/{\gamma}$ is an artificial diffusivity parameter that can be positive (the natural case \cite{CoM2017}) or negative (the metamaterial case \cite{AudibertL2017}),
and  ${\pmb \nu}$
denotes the unit outward normal vector. Without loss of generality, suppose
that $n^* = ess\sup_{\Omega} Re(n)$. In this paper, we will discuss the case where ${1}/{\gamma}>0$ but not equal to $1$.\\
\indent To describe the variational formulation of \eqref{2.1a}-\eqref{2.1d}, we first define the following vector function spaces:
\begin{eqnarray*}
\mathds{V} &=& \{(w,u) \in H^1(\Omega) \times H^1(\Omega) : w = u \text{ on } \partial \Omega\}, \\
\mathds{W} &=& L^2(\Omega) \times L^2(\Omega),
\end{eqnarray*}
\noindent equipped with the norms  
$\| {\pmb \Psi} \|_\mathds{V} = \| (\varphi,\psi) \|_\mathds{V} = \bigl( \| \varphi \|_{1}^2 + \| \psi \|_{1}^2 \bigr)^{\frac{1}{2}}, ~\forall~ {\pmb \Psi}=(\varphi,\psi)\in \bf \mathds{V}$  
and  
$\| {\pmb \Psi} \|_\mathds{W} = \bigl( \| \varphi \|_{0}^2 + \| \psi \|_{0}^2 \bigr)^{\frac{1}{2}}, ~ \forall~  {\pmb \Psi}=(\varphi,\psi)\in \mathds{W}$.
We also introduce a positive real number $\beta$ for the sake of discussion.
Using this parameter, we derive the equivalent formula for the problem \eqref{2.1a}-\eqref{2.1d} and obtain the following equations:
\begin{align}
\Delta w + k^2 n w & =  0  & \text{in}~\Omega,                   \tag{2a} \label{2.2a} \\
\frac{1}{\gamma} \Delta u + \beta k^2 u &= (\beta-\lambda) k^2 u    & \text{in}~\Omega, \tag{2b}  \label{2.2b}\\
w - u &=0   &  \text{on}~\partial \Omega, \tag{2c}  \label{2.2c}\\
\frac{\partial w}{\partial {\pmb \nu}} - \frac{1}{\gamma} \frac{\partial u}{\partial {\pmb \nu}}&=  0   & \text{on}~\partial\Omega. \tag{2d} \label{2.2d}
\end{align}
\noindent Let ${\bf U}=(w,u)$. Utilizing Green's formula, we derive the variational form of \eqref{2.2a}-\eqref{2.2d}(see \cite{CoM2017}): find $\lambda \in \mathds{C}$ and ${\bf U} \in {\bf \mathds{V}}\backslash \{ {\bf 0}\}$ such that
\setcounter{equation}{2}
\begin{equation}\label{a2.3}
  a({\bf U},{\pmb \Psi}) = (\beta-\lambda)b({\bf U},{\pmb \Psi}), \quad \forall ~ {\pmb \Psi}=(\varphi,\psi) \in {\bf \mathds{V}},
\end{equation}
where
\begin{eqnarray}
a({\bf U},{\pmb \Psi}) &:=& (\nabla w,\nabla \varphi)_0 - \frac{1}{\gamma} (\nabla u,\nabla \psi)_0 - k^2(nw,\varphi)_0 + \beta k^2(u,\psi)_0 \label{a2.4},\\
b({\bf U},{\pmb \Psi}) &:=& k^2(u,\psi)_0 \label{a2.5}.
\end{eqnarray}
Denote $\|\cdot\|_b=\sqrt{b(\cdot,\cdot)}$. \\
\indent  The adjoint problem of \eqref{a2.3} is to find $\lambda^*\in \mathds{C}$ and ${\bf U}^*=(w^*,u^*) \in {\bf \mathds{V}}\backslash \{ {\bf 0}\}$ such that
\begin{equation}\label{a2.6}
  a({\pmb \Psi},{\bf U}^*) = \overline{(\beta-\lambda^*)}b({\pmb \Psi},{\bf U}^*), \quad \forall ~ {\pmb \Psi} \in \mathds{V},
\end{equation}
and the primal and adjoint eigenvalues are connected via $(\beta-\lambda)=\overline{(\beta-\lambda^*)}$.
Clearly, the sesquilinear forms $a(\cdot,\cdot)$ and $b(\cdot,\cdot)$ are continuous. However, due to the opposing signs of the two gradient terms in $a(\cdot,\cdot)$, it is not coercive, which poses difficulties for us to conduct theoretical analysis.
To address this issue, we first introduce an isomorphism operator $\mathds{T} : \mathds{V} \rightarrow \mathds{V}$:
\begin{align*}
\mathds{T}{\pmb \Psi} =
\begin{cases}
(\varphi, 2\varphi - \psi), & \text{if } 0 < {1}/{\gamma} < 1\\
(\varphi - 2\psi, -\psi), & \text{if } {1}/{\gamma} > 1
\end{cases}
,\quad \forall ~ {\pmb \Psi}= (\varphi,\psi) \in \mathds{V}.
\end{align*}
Next, we will utilize the G{\aa}rding inequality and the operator $\mathds{T}$ to demonstrate the weak $\mathds{T}$-coercivity of $a(\cdot,\cdot)$.
\begin{lemma}\label{lem2.1}
There exists a sufficiently large positive constant $ K $ such that
\begin{eqnarray}
\mathrm{Re} (a({\pmb \Psi}, \mathds{T}{\pmb \Psi})) + K \| {\pmb \Psi} \|_{\mathds{W}}^2
\gtrsim \| {\pmb \Psi} \|_\mathds{V}^2
, \quad \forall \, {\pmb \Psi} \in \mathds{V}. \label{a2.7}
\end{eqnarray}
\end{lemma}
\begin{proof}
For the case of $0 < {1}/{\gamma} < 1$, the proof can be found in Lemma 3.1 of \cite{WangBiY2023}. For the
case of ${1}/{\gamma} > 1$, we can use the proof method of Lemma 3.1 in \cite{WangBiY2023} to obtain the desired conclusion by choosing $0 < \gamma < \varepsilon_1 < 1$ and $\frac{k^2 n^*}{K - \beta k^2} < \varepsilon_2 < \frac{K}{k^2 n^*} - 1$, where $\varepsilon_1$ and $\varepsilon_2$ are the constants in Young's inequality.  \hfill $\Box$
\end{proof}\\
\indent To study the convergence of the VE approximation for the eigenvalue problem \eqref{a2.3}, we consider the corresponding
source problem: for any \({\bf {F}} = (g,f) \in \mathds{W}\), seek ${\bf U}^{\bf F} = (w^f,u^f) \in  \mathds{V}$ such that
\begin{equation}\label{a2.8}
a({\bf U}^{\bf F},{\pmb \Psi}) = b({\bf F},{\pmb \Psi}),~ \forall ~ {\pmb \Psi} \in \mathds{V}.
\end{equation}
\noindent Firstly, we need to address the well-posedness of the problem \eqref{a2.8}, for which we work with the aid of the following auxiliary problem: seek \(k^2 \in \mathds{R}\)
and a non-zero function \({\bf U} \in \mathds{V}\) such that
\begin{equation}\label{a2.9}
  a({\bf U},{\pmb \Psi}) = 0,~ \forall ~ {\pmb \Psi} \in \mathds{V}.
\end{equation}
\noindent From \cite{Li2023, WangBiY2023} we know that if \(k^2\) is not an eigenvalue of the problem \eqref{a2.9}, then
\eqref{a2.9} only has the zero solution. Using a proof method similar to Lemma 35.3 in \cite{Ern2021}, we can derive that \(a(\cdot, \cdot)\) satisfies the inf-sup condition, ensuring that the problem \eqref{a2.8} is well-posed.
Consequently, for any \({\bf {F}} = (g,f) \in \mathds{W}\), from \eqref{a2.8} we can define the solution operator
$\mathcal{A}=(S,T):~\mathds{W} \rightarrow \mathds{V}$ by
\begin{align*}
a(\mathcal{A} {\bf F},{\pmb \Psi}) = b({\bf F},{\pmb \Psi}),~ \forall ~ {\pmb \Psi} \in \mathds{V},
\end{align*}
where
\begin{align*}
\mathcal{A} {\bf F} = (Sf,Tf)= (w^f,u^f)= {\bf U}^{\bf F}.
\end{align*}
We also denote $\mathcal{A} {\bf F} = \mathcal{A} f $, and there holds
\begin{align*}
\|\mathcal{A} {\bf F}\|_{\mathds{V}}\lesssim\|\mathbf{F}\|_{b},
\end{align*}
hence $\mathcal{A}$ is a linear bounded operator from $\mathds{W}$ to $\mathds{V}$.
Thus, since $\mathds{V}$ is compactly embedded in $\mathds{W}$,
$\mathcal{A}:~\mathds{V} \rightarrow \mathds{V}$, $\mathcal{A}:~\mathds{W} \rightarrow \mathds{W}$, and $T: L^2(\Omega) \rightarrow L^{2}(\Omega)$ are compact.
Similarly, from the source problem associated with the adjoint eigenvalue problem \eqref{a2.6} we can define the solution operator of the adjoint problem.
Specifically, for any \({\bf F}^* = (g^*,f^*) \in \mathds{W}\), the adjoint problem of \eqref{a2.8} is to find \({\bf U}^{\bf F^*} = (w^{f^*},u^{f^*}) \in \mathds{V}\) such that
\begin{equation}\label{a2.10}
a({\pmb \Phi},{\bf U}^{\bf F^*}) = b({\pmb \Phi},{\bf F^*}),~ \forall ~ {\pmb \Phi} \in \mathds{V}.
\end{equation}
And from \eqref{a2.10} we can define the
solution operator
 $\mathcal{A}^*=(S^*,T^*)~:~\mathds{W} \rightarrow \mathds{V}$ by
\begin{align*}
 a({\pmb \Phi},\mathcal{A}^* {\bf F^*}) = b({\pmb \Phi},{\bf F^*}),~ \forall ~ {\pmb \Phi} \in \mathds{V},
\end{align*}
where
\begin{align*}
\mathcal{A}^*{\bf F^*}=(S^*f^*,T^*f^*) =   (w^{f^*},u^{f^*})={\bf U}^{\bf F^*}.
\end{align*}
We also denote $\mathcal{A}^*{\bf F^*}=\mathcal{A}^*f^*$.\\
\indent Referring to the concluding remarks in \cite{WangBiY2023}, we have that for a given $f \in L^2(\Omega)$, ${\bf U}^{\bf F}
=(w^f,u^f) \in H^{1+r}(\Omega) \times H^{1+r}(\Omega)$ with $\frac{1}{2} < r \leq 1 $, and there holds
\begin{eqnarray*}
\|w^f\|_{1+r} + \|u^f\|_{1+r} \lesssim \|f\|_{0}.
\end{eqnarray*}
In the remainder of this paper, we assume that $k^2$ is not an eigenvalue of \eqref{a2.9}.

\section{The VE approximation and a priori error estimation}
\indent As a preparation for the VE discretization, we first specify some notations.
Let $\mathcal{T}_h$ be a family of non-overlapping polygonal partitions of $\Omega$. For each $E \in \mathcal{T}_h$, let $h_E$ denote the
diameter of element $E$, and let $h = \max\limits_{E \in \mathcal{T}_h} h_E$ represent the mesh size. ${\boldsymbol x}_E$ and $N_E$ respectively
denote the centroid and the number of vertices of $E$, $E_i$ represents the $i$-th vertex of $E$, and $\mathcal{E}_h$ denotes the set of all edges $e$ in $\mathcal{T}_h$.\\
\indent We also make the following assumptions: assume that there exists a positive real number \(C_\mathcal{T}\) such that
\begin{itemize}
\setlength{\itemsep}{0pt}
\setlength{\parsep}{0pt}
\setlength{\parskip}{0pt}
\item For each edge \(e \in \partial E\), its length $h_e$ satisfies \(h_e \geq C_\mathcal{T} h_E\);
\item Each element \(E\) is star-shaped with respect to all points in a disk of radius \(\geq C_\mathcal{T} h_E\).
\end{itemize}
\noindent Before discretizing the problem \eqref{a2.3} by the VEM, we first split the sesquilinear forms \(a(\cdot, \cdot)\) and \(b(\cdot, \cdot)\) as
follows:
\begin{eqnarray*}
a({\bf U}, {\pmb \Psi}) &:=& \sum_{E \in \mathcal{T}_h} a^E(w, \varphi) - \frac{1}{\gamma} a^E(u, \psi) - k^2 b^E(n w, \varphi) + \beta k^2 b^E(u, \psi),\\
b({\bf U}, {\pmb \Psi}) &:=& \sum_{E \in \mathcal{T}_h} k^2 b^E(u, \psi),\quad\quad \forall {\bf U},~{\pmb \Psi} \in \mathds{V},
\end{eqnarray*}
where
$$
a^E(\cdot, \cdot) = (\nabla \cdot, \nabla \cdot)_{0,E},~~b^E(\cdot, \cdot) = (\cdot, \cdot)_{0,E}.
$$
\indent For an element \(E\), define the boundary space 
\begin{equation*}
\mathcal{B}_l(\partial E) := \{ v: v \in C^0 (\partial E), ~v\mid_e \in \mathcal{P}_l(e),~ \forall e \in \partial E\},
\end{equation*}
where $\mathcal{P}_l(e)$ denotes the space of polynomials on $e$ with degree at most $l$.
Furthermore, for \(l \in \mathds{N}\), we denote by \(\mathcal{M}_l(E)\) the space \(\mathcal{P}_l(E)\) with an appropriate set of scaled bases, i.e.,
$$
\mathcal{M}_l(E) := \left\{ m : m = \left( \frac{{\boldsymbol x} - {\boldsymbol x}_E}{h_E} \right)^{{\boldsymbol s}}, {\boldsymbol s} \in \mathds{N}^2 \text{ and } \mid\! {\boldsymbol s}\!\mid \leq l \right\},
$$
where \({\boldsymbol s} = (s_1,s_2)\) is the multi-index, $\mid\!{\boldsymbol s}\!\mid = s_1 + s_2$, and ${\boldsymbol x}^{\boldsymbol s} = x_1^{s_1} x_2^{s_2}$.
Referring to \cite{AhmadBA2013}, we now define the elliptic projection operator \(\Pi_{l,E}^{\nabla}: H^1(E)\rightarrow \mathcal{P}_l(E)\) as follows:
\begin{eqnarray}
\begin{cases}
(\nabla v,\nabla p_l)_{0,E} = (\nabla (\Pi_{l,E}^{\nabla}v),\nabla p_l)_{0,E},~\forall p_l \in \mathcal{P}_l(E),~ v \in H^1(E),\\
P_{0}^{E}(\Pi_{l,E}^{\nabla}v) = P_{0}^{E}(v),  ~\forall v \in H^1(E), \label{a3.1}
\end{cases}
\end{eqnarray}
where
\begin{align*}
P_{0}^{E}(v) :=
\begin{cases}
\frac{1}{N_E} \sum_{i=1}^{N_E} v(E_i), & l=1,\\
\frac{1}{\mid E\mid}(v,1)_{0,E}, & l\geq 2.
\end{cases}
\end{align*}
Next, we define the local VE space
\begin{align}
\mathcal{W}^h_{l,E} := \{ v_h \in H^1(E) : & v_h \! \mid_{\partial E} \in \mathcal{B}_l(\partial E),
\Delta v_h \! \mid_E \in \mathcal{P}_l(E), \nonumber \\
& (\Pi^\nabla_{l,E}v_h - v_h, p_l)_{0,E} = 0, \forall p_l \in \mathcal{P}_l / \mathcal{P}_{l-2} (E)\}, \label{a3.2}
\end{align}
and the global VE space
\begin{eqnarray*}
\mathcal{W}^h_l := \{ v_h \in H^1(\Omega) : v_h \mid_E \in \mathcal{W}^h_{l,E}, \forall E \in \mathcal{T}_h \},
\end{eqnarray*}
where $\mathcal{P}_l / \mathcal{P}_{l-2}(E)$ denotes the polynomial subspace in $\mathcal{P}_l(E)$ that is $L^2$-orthogonal to $\mathcal{P}_{l-2}(E)$.
\noindent
We specify the degrees of freedom of \(\mathcal{W}^h_{l,E}\) as follows:
\begin{itemize}[labelwidth=!, labelindent=0pt, leftmargin=!, align=left]
\setlength{\itemsep}{0pt}
\setlength{\parsep}{0pt}
\setlength{\parskip}{0pt}
\item[$(Dof_1)$] For \(i = 1, \ldots, N_E\), the value of \(v_h\) at vertex \(E_i\);
\item[$(Dof_2)$] When \(l > 1\), for all \(e \in \partial E\), the values of \(v_h\) at \(l-1\) Gauss-Lobatto points on \(e\);
\item[$(Dof_3)$] When \(l > 1\), for all \(q_{l-2} \in \mathcal{M}_{l-2}(E)\), the moment \(\frac{1}{\mid E\mid} (v_h,q_{l-2})_{0,E}\).
\end{itemize}
\noindent Thanks to \cite{AhmadBA2013}, we know that $Ndof_{loc} := \dim(\mathcal{W}^h_{l,E}) = l \times N_E + \frac{l(l-1)}{2}$.
Furthermore, we can similarly define a computable $L^2$-projection operator $\Pi_{l,E}^{0} : L^2(E) \rightarrow \mathcal{P}_l(E)$. 
And define $\Pi_{l}^{\nabla}$ and $\Pi_{l}^{0}$ by $\Pi_{l}^{\nabla}\!\mid_{E} \! v = \Pi_{l,E}^{\nabla}v$ and $\Pi_{l}^{0}\!\mid_{E}\!v = \Pi_{l,E}^{0}v$, for all $E \in \mathcal{T}_h$, respectively. Based on $(Dof_1) - (Dof_3)$, let ${\pmb \chi}(\phi) = (\chi_1(\phi), \ldots, \chi_{Ndof_{loc}}(\phi))$ where $\chi_i(\phi)$ represents the $i$th local degree of freedom of the smooth function $\phi$. \\
\indent Now we introduce the VE approximation for the problem \eqref{a2.3}. Define the finite-dimensional space
\begin{eqnarray*}
\mathds{V}_h := \{(w_h,u_h) \in \mathcal{W}^h_{l} \times \mathcal{W}^h_{l} : (w_h-u_h) \mid_e = 0, \forall e \in \mathcal{E}_h \cap \partial\Omega \}.
\end{eqnarray*}
\noindent Let \(S_a^E(\cdot,\cdot)\) and \(S_b^E(\cdot,\cdot)\) be symmetric positive definite bilinear forms satisfying
\begin{eqnarray*}
c_0 a^E(v_h,v_h) \leq S_a^E(v_h,v_h) \leq c_1 a^E(v_h,v_h),~\forall v_h \mid_E \in \mathcal{W}^h_{l,E} ~\text{with} ~\Pi_{l,E}^{\nabla} v_h = 0, \\
c_2 b^E(v_h,v_h) \leq S_b^E(v_h,v_h) \leq c_3 b^E(v_h,v_h),~\forall v_h \mid_E \in \mathcal{W}^h_{l,E} ~\text{with} ~\Pi_{l,E}^{0} v_h = 0.
\end{eqnarray*}
\noindent As stated in \cite{Beirao2013,BeiraoBM2014}, \(S_a^E(\cdot,\cdot)\) and \(S_b^E(\cdot,\cdot)\) can be chosen as
$$
S_a^E(\varphi_h,\psi_h) = \sigma^E{\pmb \chi}(\varphi_h) \cdot {\pmb \chi}(\psi_h),~S_b^E(\varphi_h,\psi_h) = \tau^E \mid\! E \!\mid {\pmb \chi}(\varphi_h) \cdot {\pmb \chi}(\psi_h),
$$
here, $\sigma^E$ and $\tau^E$ denote the stabilization parameters. Then we define 
\begin{align*}
&a_h^E(\varphi_h,\psi_h) = a^E(\Pi_{l,E}^{\nabla} \varphi_h,\Pi_{l,E}^{\nabla} \psi_h) + S_a^E((I-\Pi_{l,E}^{\nabla})\varphi_h, (I-\Pi_{l,E}^{\nabla})\psi_h), \\
&b_h^E(n\varphi_h,\psi_h) = b^E(n\Pi_{l,E}^{0} \varphi_h, \Pi_{l,E}^{0} \psi_h) + n_0^E S_b^E((I-\Pi_{l,E}^{0})\varphi_h, (I-\Pi_{l,E}^{0})\psi_h),
\end{align*}
where $n_0^E$ is the constant approximation to $n({\boldsymbol x})$ on $E$.
\noindent Let ${\bf U}_h = (w_h,u_h)$, ${\pmb \Psi}_h = (\varphi_h,\psi_h)$, then the VE approximation of \eqref{a2.3} is
to find $({\bf U}_h,\lambda_h) \in \mathds{V}_h \backslash \{\bf 0\} \times \mathds{C}$ such that
\begin{eqnarray}
a_h({\bf U}_h,{\pmb \Psi}_h) = (\beta - \lambda_h) b_h({\bf U}_h,{\pmb \Psi}_h), \forall {\pmb \Psi}_h \in \mathds{V}_h, \label{a3.3}
\end{eqnarray}
where
\begin{align*}
a_h({\bf U}_h,{\pmb \Psi}_h):&=\sum_{E \in \mathcal{T}_h} a_h^E(w_h,\varphi_h)-\frac{1}{\gamma} a_h^E(u_h,\psi_h)-k^2 b_h^E(n w_h,\varphi_h)+\beta k^2 b_h^E(u_h,\psi_h), \\
b_h({\bf U}_h,{\pmb \Psi}_h):&= \sum_{E \in \mathcal{T}_h} k^2 b_h^E(u_h,\psi_h).
\end{align*}
The adjoint problem of \eqref{a3.3} is to find $\lambda_h^*\in \mathds{C}$ and ${\bf U}_h^*=(w_h^*,u_h^*) \in \mathds{V}_h \backslash \{ {\bf 0}\}$, such that
\begin{equation}\label{a3.4}
  a_h({\pmb \Psi},{\bf U}_h^*) = \overline{(\beta-\lambda_h^*)}b_h({\pmb \Psi},{\bf U}_h^*), \quad \forall ~ {\pmb \Psi} \in \mathds{V}_h,
\end{equation}
and the primal and adjoint eigenvalues are connected via $(\beta-\lambda_h)=\overline{(\beta-\lambda_h^*)}$.

Given ${\bf F}_h= (g_h,f_h) \in \mathds{V}_h$, the discrete source problem associated with \eqref{a3.3} is to seek ${\bf U}^{\bf F}_h = (w^f_h,u^f_h) \in \mathds{V}_h$ such that
\begin{eqnarray}
a_h({\bf U}^{\bf F}_h,{\pmb \Psi}_h) = b_h({\bf F}_h,{\pmb \Psi}_h), \forall {\pmb \Psi}_h \in \mathds{V}_h. \label{a3.5}
\end{eqnarray} 
Clearly, \(\mathcal{P}_{l}(E) \subset \mathcal{W}^h_{l,E}\). For all $E \in \mathcal{T}_h, v_h\mid_E \in \mathcal{W}^h_{l,E}$, and \(p_l \in \mathcal{P}_{l}(E)\), thanks to \cite{Beirao2013,ChenLHJ2018}, we have
\begin{align}
a_h^E(v_h,p_l) &= a^E(v_h,p_l),~b_h^E(v_h,p_l) = b^E(v_h,p_l), \label{a3.6} \\
\alpha_* a^E(v_h,v_h) &\leq a_h^E(v_h,v_h) \leq \alpha^* a^E(v_h,v_h), \label{a3.7} \\
c_* b^E(v_h,v_h) &\leq b_h^E(v_h,v_h) \leq c^* b^E(v_h,v_h). \label{a3.8}
\end{align}
\noindent Since \(\mathds{V}_h \subset \mathds{V}\), we also have the weak $\mathds{T}$-coercivity of $a_h(\cdot,\cdot)$.
\begin{lemma}\label{lem3.1} There exists a sufficiently large positive constant \(K\) such that
\begin{eqnarray}
\mathrm{Re}(a_h({\pmb \Psi}_h, \mathds{T}{\pmb \Psi}_h)) + K\| {\pmb \Psi}_h \|_{\mathds{W}}^2
\gtrsim \| {\pmb \Psi}_h \|_{\mathds{V}}^2
,~ \forall ~ {\pmb \Psi}_h \in \mathds{V}_h. \label{a3.9}
\end{eqnarray}
\end{lemma}
\begin{proof}
When $0 < {1}/{\gamma} < 1$, from the definitions of the operator \(\mathds{T}\) and \(a_h(\cdot, \cdot)\), along with
\eqref{a3.7} and \eqref{a3.8}, we deduce that
\begin{align*}
& \text{Re}(a_h({\pmb \Psi}_h,\mathds{T}{\pmb \Psi}_h))+ K\| {\pmb \Psi}_h \|_{\mathds{W}}^2 \\
& =\sum_{E \in \mathcal{T}_h} \text{Re} \left( a_h^E(\varphi_h,\varphi_h) -\frac{1}{\gamma} a_h^E(\psi_h,2\varphi_h-\psi_h)- k^2 b_h^E(n \varphi_h,\varphi_h)
+  \beta k^2 b_h^E(\psi_h,2\varphi_h-\psi_h) \right)\\
&\quad + \sum_{E \in \mathcal{T}_h} K(\varphi_h,\varphi_h)_E + K(\psi_h,\psi_h)_E \\
& =\sum_{E \in \mathcal{T}_h} \text{Re} \left( a_h^E(\varphi_h,\varphi_h) - 2 \frac{1}{\gamma} a_h^E(\psi_h,\varphi_h)+\frac{1}{\gamma} a_h^E(\psi_h,\psi_h)
- k^2 b_h^E(n \varphi_h,\varphi_h) \right) \\
& \quad + \sum_{E \in \mathcal{T}_h}  \text{Re}(2\beta k^2 b_h^E(\psi_h,\varphi_h)) +  K(\varphi_h,\varphi_h)_E + K(\psi_h,\psi_h)_E -\beta k^2 b_h^E(\psi_h,\psi_h)\\
& \geq \sum_{E \in \mathcal{T}_h} \alpha_* a^E(\varphi_h,\varphi_h) - 2 \frac{1}{\gamma} \text{Re} (a_h^E(\psi_h,\varphi_h)) 
+ \frac{1}{\gamma} \alpha_* a^E(\psi_h,\psi_h)  + 2\beta k^2 \text{Re}(b_h^E(\psi_h,\varphi_h)) \\
& \quad + \sum_{E \in \mathcal{T}_h}   K \|\varphi_h\|_{0,E}^2 - k^2 \max\{n^*,n_0^E\} c^* b^E(\varphi_h,\varphi_h)
+ K \|\psi_h\|_{0,E}^2 -\beta k^2 c^* b^E(\psi_h,\psi_h)\\
& \geq \sum_{E \in \mathcal{T}_h} \alpha_* \|\nabla \varphi_h\|_{0,E}^2 + \frac{1}{\gamma} \alpha_* \|\nabla \psi_h\|_{0,E}^2 +
(K-k^2 \max\{n^*,n_0^E\} c^*) \|\varphi_h\|_{0,E}^2   \\
&\quad + \sum_{E \in \mathcal{T}_h} (K- \beta k^2 c^*)\|\psi_h\|_{0,E}^2 - 2 \frac{1}{\gamma}  \mid a_h^E(\psi_h,\varphi_h)\mid
- 2 \beta k^2 \mid b_h^E(\psi_h,\varphi_h)\mid.
\end{align*}
\noindent By Young's inequality, for all \(\varepsilon_3, \varepsilon_4 > 0\), we have
$2\|\psi_h\|_{0,E} \|\varphi_h\|_{0,E} \leq \varepsilon_3 \|\psi_h\|_{0,E}^2 + \varepsilon_3^{-1} \|\varphi_h\|_{0,E}^2$
and
$2\|\nabla \psi_h\|_{0,E} \|\nabla \varphi_h\|_{0,E} \leq \varepsilon_4 \|\nabla \psi_h\|_{0,E}^2 + \varepsilon_4^{-1} \|\nabla \varphi_h\|_{0,E}^2$,
then, by the Cauchy-Schwarz inequality, we have
\begin{align*}
\text{Re}&(a_h({\pmb \Psi}_h,\mathds{T}{\pmb \Psi}_h)) + K\| {\pmb \Psi}_h \|_{\mathds{W}}^2\\
\geq &  \sum_{E \in \mathcal{T}_h} \alpha_* \|\nabla \varphi_h\|_{0,E}^2 + \frac{1}{\gamma} \alpha_* \|\nabla \psi_h\|_{0,E}^2 +
(K-k^2 \max\{n^*,n_0^E\} c^*) \|\varphi_h\|_{0,E}^2   \\
& + \sum_{E \in \mathcal{T}_h} (K- \beta k^2  c^*)\|\psi_h\|_{0,E}^2 - 2 \frac{1}{\gamma} \alpha^* \|\nabla \psi_h\|_{0,E} \|\nabla \varphi_h\|_{0,E} - 2\beta k^2  c^* \|\psi_h\|_{0,E} \| \varphi_h\|_{0,E} \\
\geq &  \sum_{E \in \mathcal{T}_h} (\alpha_*-\frac{1}{\gamma} \alpha^* \varepsilon_4^{-1})\|\nabla \varphi_h\|_{0,E}^2 + \frac{1}{\gamma} (\alpha_*-\alpha^* \varepsilon_4) \|\nabla \psi_h\|_{0,E}^2 \\
&+ \sum_{E \in \mathcal{T}_h} (K-k^2 \max\{n^*,n_0^E\} c^* - \beta k^2  c^* \varepsilon_3^{-1}) \|\varphi_h\|_{0,E}^2 +  (K- \beta k^2 c^* - \beta k^2 c^* \varepsilon_3)\|\psi_h\|_{0,E}^2.
\end{align*}
Choose $\frac{\beta k^2  c^*}{K-k^2 \max\{n^*,n_0^E\} c^*} < \varepsilon_3 < \frac{K- \beta k^2  c^*}{\beta k^2  c^*}$
and $\frac{1}{\gamma} < \frac{\alpha^*}{\gamma \alpha_*} < \varepsilon_4 < \frac{\alpha_*}{\alpha^*} < 1$,
and we get \eqref{a3.9} immediately. Similarly, we can prove the desired result in the case of ${1}/{\gamma} >1$. The proof is completed.  \hfill $\Box$
\end{proof}

\indent Consider now the VE approximation of \eqref{a2.9}: seek \(({\bf U}_h,k_h^2) \in \mathds{V}_h \backslash \{ \bf 0\}  \times \mathds{R}\)
such that
\begin{eqnarray}
a_h({\bf U}_h,{\pmb \Psi}_h)=0,\quad \forall {\pmb \Psi}_h \in \mathds{V}_h.\label{a3.10}
\end{eqnarray}
Since \(k^2\) is not an eigenvalue of \eqref{a2.9}, it is known that for sufficiently small \(h\), \(k_h^2\) is also
not an eigenvalue of \eqref{a3.10} and $k_h^2$ converges to $k^2$ (see \cite{MengMei2022}).
Therefore, from \eqref{a3.9} and the fact that $\mathds{T}$ is an isomorphism (noting that $\mathds{T}^2 =\mathds{I}$), by Fredholm's alternative we conclude that the problem \eqref{a3.5} is well-posed.
Hence, for ${\bf F}_h=({g_h,f_h}) \in \mathds{V}_h$,  we can define the corresponding discrete solution operator
$\widetilde{\mathcal{A}}_h=(\widetilde{S}_h,\widetilde{T}_h):~\mathds{V}_h \rightarrow \mathds{V}_h$ by
\begin{align*}
 a_h(\widetilde{\mathcal{A}}_h {\bf F}_h,{\pmb \Psi}_h) = b_h({\bf F}_h,{\pmb \Psi}_h),~ \forall ~ {\pmb \Psi}_h \in \mathds{V}_h,
\end{align*}
where
\begin{align*}
 \widetilde{\mathcal{A}}_h {\bf F}_h=(\widetilde{S}_h f_h,\widetilde{T}_h f_h)=(w^f_h,u^f_h)= {\bf U}^{\bf F}_h.
\end{align*}
We also denote $ \widetilde{\mathcal{A}}_h {\bf F}_h=\widetilde{\mathcal{A}}_h f_h$.
Additionally, it is valid that
\begin{align}
\|\widetilde{\mathcal{A}}_h {\bf F}_h\|_{\mathds{V}}\lesssim \|{\bf F}_h\|_b. \label{a3.11}
\end{align} 
Given ${\bf F}_h^* = (g_h^*,f_h^*) \in \mathds{V}_h $, the VE approximation of \eqref{a2.10} is to seek ${\bf U}^{\bf F^*}_h =$ $ (w^{f^*}_h,u^{f^*}_h) \in \mathds{V}_h$ such that
\begin{equation}\label{a3.12}
a_h({\pmb \Phi}_h,{\bf U}^{\bf F^*}_h) = b_h({\pmb \Phi}_h,{\bf F}_h^*),~ \forall ~ {\pmb \Phi}_h \in \mathds{V}_h.
\end{equation}
Similarly,  we can define the corresponding discrete adjoint solution operator $\widetilde{\mathcal{A}}_h^*=(\widetilde{S}_h^*,\widetilde{T}_h^*):~\mathds{V}_h  \rightarrow \mathds{V}_h$ by
\begin{align*}
&a_h({\pmb \Phi}_h,\widetilde{\mathcal{A}}^*_h {\bf F}^*_h) = b_h({\pmb \Phi}_h,{\bf F}^*_h),~ \forall ~ {\pmb \Phi}_h \in \mathds{V}_h, \\
&\widetilde{\mathcal{A}}^*_h {\bf F}^*_h=(\widetilde{S}_h^*f_h^*,\widetilde{T}_h^*f_h^*) = (w^{f^*}_h,u^{f^*}_h)= {\bf U}^{\bf F^*}_h.
\end{align*}
We also denote $\widetilde{\mathcal{A}}^*_h {\bf F}^*_h=\widetilde{\mathcal{A}}^*_h f_h^*$.\\
\indent To obtain a priori error estimates for eigenfunctions and eigenvalues using spectral approximation theory, we need to show that
$\| \widetilde{T}_h - T \|_{L^2(\Omega) \rightarrow L^2(\Omega)} \to 0$ as $ h \to 0$
and
$\|\widetilde{\mathcal{A}}_h - \mathcal{A} \|_{\mathds{V} \rightarrow \mathds{V}} \to 0$ as $h \to 0.$
However, $\widetilde{\mathcal{A}}_h$ is not well defined on $\mathds{V}$. To overcome this drawback, inspired by the work \cite{MoraR2020}, we introduce the projection operator
$\mathbf{P}_h=(P_h,P_h): \mathds{W} \to \mathds{V}_h \subset \mathds{W}$
with range $\mathds{V}_h$, defined by the property
$b(\mathbf{P}_h \mathbf{F} - \mathbf{F}, {\pmb \Psi}) = 0, ~ \forall {\pmb \Psi} \in \mathds{V}_h.$
Moreover, we have 
\begin{eqnarray}
\|P_h f\|_0 \leq \|f\|_0, \quad 
\|P_h f - f\|_0 = \inf_{f_h \in \mathcal{W}^h_{l}} \|f - f_h\|_0. \label{aa3.27}
\end{eqnarray}
With the aid of this operator, we introduce the following operator
$\mathcal{A}_h =(S_h,T_h):=  \widetilde{\mathcal{A}}_h \mathbf{P}_h: \mathds{W} \to \mathds{V}_h,$
which is well defined for any source term $\mathbf{F} \in \mathds{W}$. Furthermore, it is easy to verify that $\widetilde{\mathcal{A}}_h$ and $\mathcal{A}_h$ share the same nonzero spectrum and eigenfunctions. Following a similar approach, we can define the adjoint auxiliary solution operator $\mathcal{A}_h^*=(S_h^*,T_h^*)$.\\
\indent Next, we will derive the a priori error estimates for the VE approximation. From \cite{Beirao2013,Antonietti2022,ChenLHJ2018} we have the following projection error estimates and interpolation error estimates (For example, see the details in Proposition 7.1 in \cite{Antonietti2022}):
\noindent For every $ v \in H^t(\Omega) $, with $ 1 \leq t \leq l+1 $, there exists $ v_{\pi} \in \mathcal{P}_l(E) $ such that
\begin{equation}
\| v - v_\pi\|_{0,E} + h_E \mid v - v_\pi \mid_{1,E} \leq C h_E^t \mid v\mid_{t,E}; \label{a3.13}
\end{equation}
Moreover, for every $ v \in H^t(\Omega) $, with $ 1 \leq t \leq l+1 $, there exists $ v_I \in \mathcal{W}_l^h $ such that
\begin{equation}
\| v - v_I\|_{0,E} + h_E \mid v - v_I \mid_{1,E} \leq C h_E^t \mid v\mid_{t,E}. \label{a3.14}
\end{equation}
Notice that $v_\pi$ is defined element by element, and does not belong to the space $ H^1(\Omega) $. We shall denote its broken $H^1$ norm by $\| v_\pi \|_{1,h}:=\left(\sum_{E \in \mathcal{T}_h} \| v_\pi \|_{1,E}^2 \right)^{1/2}$.
\begin{lemma}\label{lem3.2}
Suppose $n$ is a piecewise smooth function.
For each ${\pmb \Psi}=(\varphi,\psi) \in \mathbf{\mathds{V}}$, let ${\pmb \Psi}_\pi = (\varphi_\pi, \psi_\pi)$ with $\varphi_\pi\mid_E \in \mathcal{P}_l(E)$ and $\psi_\pi\mid_E \in \mathcal{P}_l(E)$ satisfying \eqref{a3.13}.
Then, it is valid for any ${\pmb \Psi}_h \in \mathbf{\mathds{V}}_h$ that
\begin{align}
a_h({\pmb \Psi}_h, {\pmb \Psi}_\pi) - a({\pmb \Psi}_h, {\pmb \Psi}_\pi)
\lesssim \| n \varphi_\pi - \Pi_{l}^{0} (n \varphi_\pi) \|_{0}\| \varphi_h - \Pi_{l}^{0} \varphi_h\|_{0},\label{a3.15} \\
a_h({\pmb \Psi}_\pi,{\pmb \Psi}_h) - a({\pmb \Psi}_\pi,{\pmb \Psi}_h)
\lesssim \| n \varphi_\pi - \Pi_{l}^{0} (n \varphi_\pi) \|_{0}\| \varphi_h - \Pi_{l}^{0} \varphi_h\|_{0}. \label{pa3.15} 
\end{align}
\end{lemma}
\begin{proof}
Utilizing the definitions of $a_h(\cdot,\cdot)$ and $a(\cdot,\cdot)$, along with \eqref{a3.6}, we find
\begin{align*}
&a_h({\pmb \Psi}_h,{\pmb \Psi}_\pi) - a({\pmb \Psi}_h,{\pmb \Psi}_\pi) \\
&= \sum_{E \in \mathcal{T}_h} a_h^E(\varphi_h,\varphi_\pi) - \frac{1}{\gamma} a_h^E(\psi_h,\psi_\pi) - k^2 b_h^E(n \varphi_h,\varphi_\pi) + \beta k^2 b_h^E(\psi_h,\psi_\pi)\\
&\quad -  \sum_{E \in \mathcal{T}_h} a^E(\varphi_h,\varphi_\pi) + \frac{1}{\gamma} a^E(\psi_h,\psi_\pi) + k^2 b^E(n \varphi_h,\varphi_\pi) - \beta k^2 b^E(\psi_h,\psi_\pi)\\
&=  \sum_{E \in \mathcal{T}_h} - k^2 b_h^E(n \varphi_h,\varphi_\pi) + k^2 b^E(n \varphi_h,\varphi_\pi),
\end{align*}
then, according to the definition of $b_h^E(\cdot,\cdot)$ and a proof similar to Lemma 5.3 in \cite{BeirVBM2016}, we derive
\begin{align*}
&a_h({\pmb \Psi}_h,{\pmb \Psi}_\pi) - a({\pmb \Psi}_h,{\pmb \Psi}_\pi) \\
&= k^2\sum_{E \in \mathcal{T}_h}  b^E(n \varphi_h,\varphi_\pi) - b^E(n\Pi_{l,E}^{0} \varphi_h, \Pi_{l,E}^{0} \varphi_\pi)\\
&\lesssim  \| n \varphi_h - \Pi_{l}^{0} (n \varphi_h) \|_{0}\| \varphi_\pi - \Pi_{l}^{0} \varphi_\pi\|_{0}
+ \| n \varphi_\pi - \Pi_{l}^{0} (n \varphi_\pi) \|_{0}\| \varphi_h - \Pi_{l}^{0} \varphi_h\|_{0} \\
&~~~+  \| \varphi_\pi - \Pi_{l}^{0} \varphi_\pi \|_{0}\| \varphi_h - \Pi_{l}^{0} \varphi_h\|_{0},
\end{align*}
where $\| \varphi_\pi - \Pi_{l}^{0} \varphi_\pi \|_{0}=0$, then we get \eqref{a3.15}. Using a similar argument, \eqref{pa3.15} can be proved. The proof is completed. \hfill $\Box$
\end{proof}
\begin{lemma}\label{lem3.3} Assuming that the conditions of Lemma $\ref{lem3.2}$ hold. Let ${\bf U}^{\bf F}=(w^f,u^f)$ and ${\bf U}^{\bf F}_h:= \mathcal{A}_h \mathbf{F} = \widetilde{\mathcal{A}}_h \mathbf{P}_h \mathbf{F} $ be the solution of \eqref{a2.8} and \eqref{a3.5}, respectively. Then there holds
\begin{align}
&\| {\bf U}^{\bf F} - {\bf U}^{\bf F}_h \|_{\mathds{W}} \lesssim  h^{r} \{ \|{\bf U}^{\bf F} - {\bf U}^{\bf F}_h\|_{\mathds{V}} + \|w^f - w_\pi^f\|_{1,h} + \|u^f - u_\pi^f\|_{1,h} + \nonumber \\
& ~h\{ \|f-f_I\|_0 +\|f-\Pi_{l}^{0}f\|_{0} + \| w^f- \Pi_{l}^{0}w^f \|_{0} +  \| n w^f- \Pi_{l}^{0}(n w^f) \|_{0}  \}\}, \label{a3.16}
\end{align}
where $w_\pi^f$ and $u_\pi^f$ are the approximations of $w^f$ and $u^f$ satisfying \eqref{a3.13}, respectively, and $f_I$ is the interpolation of $f$ satisfying \eqref{a3.14}.
\end{lemma}
\begin{proof}We use the duality argument to complete the proof. First, introduce the following auxiliary problem: for any \({\bf \widehat{G}} = (\hat{f},\hat{g}) \in \mathds{W}\), find \({\pmb \Phi}=(\phi,z) \in \mathds{V}\) such that
\begin{eqnarray*}
a({\pmb \Psi},{\pmb \Phi})=B({\bf \widehat{G}},{\pmb \Psi}), \quad \forall {\pmb \Psi} \in \mathds{V},
\end{eqnarray*}
where \(B({\bf \widehat{G}},{\pmb \Psi})=\int_\Omega {\bf \widehat{G}}\cdot {\boldsymbol {\overline{\Psi}}}\). Since \(k^2\) is not an eigenvalue of \eqref{a2.9},
by Fredholm's alternative we know the above auxiliary problem has a unique solution \({\pmb \Phi}\), and
$$\|\phi\|_{1+r} + \|z\|_{1+r} \lesssim \|{\bf \widehat{G}}\|_{\mathds{W}}~~(\frac{1}{2} < r \leq 1).$$
Let \({\pmb \Psi}={\bf U}^{\bf F} - {\bf U}^{\bf F}_h\), then
\begin{eqnarray*}
B({\bf \widehat{G}},{\bf U}^{\bf F} - {\bf U}^{\bf F}_h)= a({\bf U}^{\bf F} - {\bf U}^{\bf F}_h,{\pmb \Phi}) = a({\bf U}^{\bf F} - {\bf U}^{\bf F}_h,{\pmb \Phi} - {\pmb \Phi}_I ) + a({\bf U}^{\bf F} - {\bf U}^{\bf F}_h,{\pmb \Phi}_I ),
\end{eqnarray*}
where \({\pmb \Phi}_I=(\phi_I,z_I) \in \mathbf{\mathds{V}}_h \) is the interpolation of \({\pmb \Phi}\).
Using the continuity of \(a(\cdot, \cdot)\) and \eqref{a3.14}, we obtain
\begin{eqnarray}
a({\bf U}^{\bf F} - {\bf U}^{\bf F}_h,{\pmb \Phi} - {\pmb \Phi}_I )
 \lesssim \|{\pmb \Phi} - {\pmb \Phi}_I\|_{\mathds{V}} \|{\bf U}^{\bf F} - {\bf U}^{\bf F}_h\|_{\mathds{V}} \lesssim h^r \|{\bf \widehat{G}}\|_{\mathds{W}} \|{\bf U}^{\bf F} - {\bf U}^{\bf F}_h\|_{\mathds{V}}. \label{a3.17}
\end{eqnarray}
\noindent From \eqref{a2.8} and the definition of $\mathcal{A}_h$, we have 
\begin{eqnarray*}
 a({\bf U}^{\bf F} - {\bf U}^{\bf F}_h,{\pmb \Phi}_I )  = \{ b({\bf F},{\pmb \Phi}_I)-b_h(\mathbf{P}_h {\bf F},{\pmb \Phi}_I) \} +
 \{ a_h({\bf U}^{\bf F}_h,{\pmb \Phi}_I) - a({\bf U}^{\bf F}_h,{\pmb \Phi}_I) \}.
\end{eqnarray*}
\noindent According to the definition of $b_h^E(\cdot,\cdot)$ and $\mathbf{P}_h$, \eqref{a3.6}, \eqref{a3.13}, \eqref{a3.14} and \eqref{aa3.27}, we deduce that
\begin{eqnarray}
 && b({\bf F},{\pmb \Phi}_I)-b_h({\mathbf{P}_h {\bf F}},{\pmb \Phi}_I) \nonumber\\
 && \lesssim  \sum_{E \in \mathcal{T}_h} b^E(P_hf,z_I) - b_h^E(P_hf,z_I)\nonumber\\
 && \lesssim  \sum_{E \in \mathcal{T}_h} b^E(P_h f-\Pi_{l,E}^{0}f,z_I-\Pi_{l,E}^{0}z_{I}) - b_h^E(P_hf-\Pi_{l,E}^{0}f,z_I-\Pi_{l,E}^{0}z_{I})\nonumber\\
 && \lesssim  h^{1+r} \|P_h f-\Pi_{l}^{0}f\|_{0} \|{\bf \widehat{G}}\|_{\mathds{W}} \nonumber\\
 && \lesssim  h^{1+r} \{ \|f-f_I\|_{0}+\|f-\Pi_{l}^{0}f\|_{0} \}\|{\bf \widehat{G}}\|_{\mathds{W}}, \label{a3.18}
\end{eqnarray}
where $f_I$ is the interpolation of $f$ that satisfies \eqref{a3.14}.\\
\indent Let ${\pmb \Phi}_\pi$ and ${\bf U}^{\bf F}_{\pi}$ be the approximation of ${\pmb \Phi}$ and ${\bf U}^{\bf F}$ that satisfy \eqref{a3.13}, respectively. Using \eqref{a3.6}, we derive that
\begin{align*}
a_h({\bf U}^{\bf F}_h,{\pmb \Phi}_I) - a({\bf U}^{\bf F}_h,{\pmb \Phi}_I) 
&= a_h({\bf U}^{\bf F}_h-{\bf U}^{\bf F}_{\pi},{\pmb \Phi}_I-{\pmb \Phi}_\pi) - a({\bf U}^{\bf F}_h-{\bf U}^{\bf F}_{\pi},{\pmb \Phi}_I-{\pmb \Phi}_\pi) \\
 &\quad+ a_h({\bf U}^{\bf F}_{h},{\pmb \Phi}_\pi) - a({\bf U}^{\bf F}_{h},{\pmb \Phi}_\pi)
 +a_h({\bf U}^{\bf F}_\pi,{\pmb \Phi}_I) - a({\bf U}^{\bf F}_\pi,{\pmb \Phi}_I).
\end{align*}
Using the Cauchy-Schwarz inequality, \eqref{a3.7}, \eqref{a3.8}, \eqref{a3.15}, \eqref{pa3.15}, \eqref{a3.13}, \eqref{a3.14}, and the continuity of $a(\cdot,\cdot)$, we derive
\begin{align}
&a_h({\bf U}^{\bf F}_h,{\pmb \Phi}_I) - a({\bf U}^{\bf F}_h,{\pmb \Phi}_I)\nonumber\\
& \lesssim  \{ \|{\bf U}^{\bf F} - {\bf U}^{\bf F}_h\|_{\mathds{V}} + \|w^f - w_\pi^f\|_{1,h} + \|u^f - u_\pi^f\|_{1,h} \} \|{\pmb \Phi}_I-{\pmb \Phi}_\pi\|_\mathds{V} \nonumber\\
&\quad + \| n \phi_\pi- \Pi_{l}^{0} (n \phi_\pi) \|_{0} \| w_h^f- \Pi_{l}^{0}w_h^f \|_{0} 
+ \| n w_\pi^f- \Pi_{l}^{0} (n w_\pi^f) \|_{0} \| \phi_I- \Pi_{l}^{0} \phi_I \|_{0} \nonumber\\
& \lesssim h^{r} \{ \|{\bf U}^{\bf F} - {\bf U}^{\bf F}_h\|_{\mathds{V}} + \|w^f - w_\pi^f\|_{1,h} + \|u^f - u_\pi^f\|_{1,h} \} \|{\bf \widehat{G}}\|_\mathds{W} \nonumber\\
&\quad+ h^{1+r} \{\| w^f- \Pi_{l}^{0}w^f \|_{0} + \| n w^f- \Pi_{l}^{0}(n w^f) \|_{0} \}\|{\bf \widehat{G}}\|_{\mathds{W}},\label{a3.19} 
\end{align}
then, by combining \eqref{a3.17}, \eqref{a3.18}, and \eqref{a3.19}, we obtain
\begin{align*}
\| {\bf U}^{\bf F} - {\bf U}^{\bf F}_h \|_{\mathds{W}} &= \sup_{{\bf 0} \neq {\bf \widehat{G}} \in \mathds{W}}
\frac{\mid B({\bf \widehat{G}},{\bf U}^{\bf F} - {\bf U}^{\bf F}_h)\mid}{\|{\bf \widehat{G}}\|_{\mathds{W}}}\\
&\lesssim  h^{r} \{ \|{\bf U}^{\bf F} - {\bf U}^{\bf F}_h\|_{\mathds{V}} + \|w^f - w_\pi^f\|_{1,h} + \|u^f - u_\pi^f\|_{1,h}+ h \{\|f-f_I\|_0  \\
&\quad + \|f-\Pi_{l}^{0}f\|_{0} + \| w^f- \Pi_{l}^{0}w^f \|_{0} + \| n w^f- \Pi_{l}^{0}(n w^f) \|_{0} \}\} ,
\end{align*}
i.e., \eqref{a3.16} is valid. The proof is completed.  \hfill $\Box$
\end{proof}
\begin{remark}
When \(n\) is a constant or piecewise constant, \eqref{a3.16} simplifies to
\begin{align*}
\| {\bf U}^{\bf F} - {\bf U}^{\bf F}_h \|_{\mathds{W}} & \lesssim  h^{r} \{ \|{\bf U}^{\bf F} - {\bf U}^{\bf F}_h\|_{\mathds{V}} + \|w^f - w_\pi^f\|_{1,h} + \|u^f - u_\pi^f\|_{1,h} \\
&\quad + h(\|f-f_I\|_0 +\|f-\Pi_{l}^{0}f\|_0) \}.
\end{align*}
\end{remark}
\begin{lemma}\label{lem3.4}
Suppose that $f\in H^{1+\iota}(\Omega)~(-1\leq \iota \leq s)$. Let ${\bf U}^{\bf F}=(w^f,u^f)$ and ${\bf U}^{\bf F}_h=\mathcal{A}_h \mathbf{F} = \widetilde{\mathcal{A}}_h \mathbf{P}_h \mathbf{F} $ be the solution of \eqref{a2.8} and \eqref{a3.5},
respectively, and ${\bf U}^{\bf F} \in H^{1+s}(\Omega) \times H^{1+s}(\Omega)~( r\leq s \leq l )$.  Then there holds 
\begin{eqnarray}
 \| {\bf U}^{\bf F} - {\bf U}^{\bf F}_h \|_{\mathds{V}} \lesssim  h^s \{ \| w^f \|_{1+s} + \| u^f \|_{1+s} \}  + {h\|f-\Pi_{l}^{0}f\|_{0}}+ h\|f-f_I\|_0.\label{a3.20}
\end{eqnarray}
\end{lemma}
\begin{proof}
Denote $\hat{w}_h = w^f_h - w^f_I$ and $\hat{u}_h = u^f_h - u^f_I$, where $(w^f_I,u^f_I)={\bf U}^{\bf F}_I$ is the interpolations of ${\bf U}^{\bf F}$. Then, from the triangle inequality, we obtain
\begin{eqnarray*}
\| {\bf U}^{\bf F} - {\bf U}^{\bf F}_h \|_{\mathds{V}} \leq  \| {\bf U}^{\bf F} - {\bf U}^{\bf F}_I \|_{\mathds{V}} +  \{ \| \hat{w}_h \|_{1}^2 +   \| \hat{u}_{h} \|_{1}^2 \}^{1/2} := I + I\!I.
\end{eqnarray*}
For the first part $I$, based on the interpolation error estimate \eqref{a3.14}, we obtain
\begin{eqnarray*}
I = \{ \| w^f-w^f_I \|_{1}^2 + \| u^f-u^f_I \|_{1}^2 \}^{1/2}  \lesssim  h^s \{ \| w^f \|_{1+s} + \| u^f \|_{1+s} \} .
\end{eqnarray*}
For the second part $I\!I$, assuming $0 < {1}/{\gamma} < 1$, and using \eqref{a3.9}, we have
\begin{align*}
&I\!I^2 = \| \hat{w}_h \|_{1}^2 +   \| \hat{u}_{h} \|_{1}^2 \\
& \leq C\{\text{Re} (a_h((\hat{w}_h,\hat{u}_{h}),\mathds{T}(\hat{w}_h,\hat{u}_{h}))) + K \| \hat{w}_h \|_{0}^2 + K \| \hat{u}_h \|_{0}^2 \}\\
& = C \{  \text{Re} (a_h((\hat{w}_h,\hat{u}_{h}),(\hat{w}_h,2\hat{w}_h-\hat{u}_{h}))) + K \| \hat{w}_h \|_{0}^2 + K \| \hat{u}_h \|_{0}^2 \}\\
& = C \{  \text{Re} (2 a_h((\hat{w}_h,\hat{u}_{h}),(\hat{w}_h,\hat{w}_h))-a_h((\hat{w}_h,\hat{u}_{h}),(\hat{w}_h,\hat{u}_{h}))) + K \| \hat{w}_h \|_{0}^2 + K \| \hat{u}_h \|_{0}^2 \}\\
& = C \{ 2 \text{Re}\{ a_h(({w}^f_h,{u}^f_h),(\hat{w}_h,\hat{w}_h)) - a_h(({w}^f_I,{u}^f_I),(\hat{w}_h,\hat{w}_h)) \} \\
&\quad - \text{Re} (a_h(({w}^f_h,{u}^f_h),(\hat{w}_h,\hat{u}_{h}))) 
  + \text{Re} (a_h(({w}^f_I,{u}^f_I),(\hat{w}_h,\hat{u}_{h}))) + K \| \hat{w}_h \|_{0}^2 + K \| \hat{u}_h \|_{0}^2 \}.
\end{align*}
Furthermore, from Lemma \ref{lem3.3}, we know that $\{ \| \hat{w}_h \|_{0}^2 +  \| \hat{u}_h \|_{0}^2 \}$ is a higher-order term compared
with $\{\| \hat{w}_h \|_{1}^2 +  \| \hat{u}_h \|_{1}^2 \}$. According to \eqref{a2.8} and the definition of $\mathcal{A}_h$, and from \eqref{a3.13} we have
\begin{eqnarray*}
I\!I^2 &\lesssim & \left\{ \| \hat{w}_h \|_{1} + \| \hat{u}_{h} \|_{1} \right \}^2 \\
&\lesssim & 2 \text{Re} \{a_h(({w}^f_h,{u}^f_h),(\hat{w}_h,\hat{w}_h)) - a_h(({w}^f_I,{u}^f_I),(\hat{w}_h,\hat{w}_h)) \} \\
&~&+~ \text{Re} \{ a_h(({w}^f_I,{u}^f_I),(\hat{w}_h,\hat{u}_{h}))
- a_h(({w}^f_h,{u}^f_h),(\hat{w}_h,\hat{u}_{h})) \},
\end{eqnarray*}
where
\begin{eqnarray*}
&& 2\text{Re} \{a_h(({w}^f_h,{u}^f_h),(\hat{w}_h,\hat{w}_h)) - a_h(({w}^f_I,{u}^f_I),(\hat{w}_h,\hat{w}_h)) \} \\
&& \lesssim  \text{Re} \left\{ b_h((0,P_h f),(\hat{w}_h,\hat{w}_h)) - a_h(({w}^f_I-w^f_\pi,{u}^f_I-u^f_\pi),(\hat{w}_h,\hat{w}_h))  \right. \\
&& ~~~-  \left. a_h((w^f_\pi,u^f_\pi),(\hat{w}_h,\hat{w}_h)) - b((0,f),(\hat{w}_h,\hat{w}_h)) + a(({w}^f,{u}^f),(\hat{w}_h,\hat{w}_{h})) \right\}\\
&&\lesssim   \text{Re} \left\{b_h((0,P_h f),(\hat{w}_h,\hat{w}_h)) - b((0,f),(\hat{w}_h,\hat{w}_h))  \right\}   \\
&& ~~~-  \text{Re} \{  a_h(({w}^f_I-w^f_\pi,{u}^f_I-u^f_\pi),(\hat{w}_h,\hat{w}_h)) + a((w^f_\pi-{w}^f,u^f_\pi-{u}^f),(\hat{w}_h,\hat{w}_{h}))  \\
&& ~~~   +a_h((w^f_\pi,u^f_\pi),(\hat{w}_h,\hat{w}_h)) - a((w^f_\pi,u^f_\pi),(\hat{w}_h,\hat{w}_{h})) \} \\
&&:= I^{\prime} - I\!I^{\prime}.
\end{eqnarray*}
Next, we will estimate $I^{\prime}$ and $I\!I^{\prime}$ separately.
For $I^{\prime}$, 
using the definition of $\mathbf{P}_h$, \eqref{a3.13}, \eqref{a3.6}, \eqref{aa3.27} and the Cauchy-Schwarz inequality, similarly to the argument in \eqref{a3.18}, then
\begin{eqnarray*}
I^{\prime}  =  k^2 \sum_{E \in \mathcal{T}_h}  \text{Re} (b_h^E(P_h f, \hat{w}_h) - b^E(f, \hat{w}_h) )
 \lesssim  h (\|f-f_I\|_0+\|f-\Pi_{l}^0f\|_0) \| \hat{w}_h \|_1.
\end{eqnarray*}
For $I\!I^{\prime}$, using the Cauchy-Schwarz inequality, \eqref{a3.7}, \eqref{a3.8}, \eqref{pa3.15}, \eqref{a3.13} and \eqref{a3.14}, similarly to the argument in \eqref{a3.19}, we have
\begin{eqnarray*}
I\!I^{\prime} & \lesssim & 
\| \hat{w}_h \|_{1}  \left\{ \|w^f - {w}^f_I\|_{1} + \|  u^f - {u}^f_I \|_{1} + \| w^f_\pi - {w}^f \|_{1,h} + \| u^f_\pi - {u}^f \|_{1,h} \right\}\\
& \quad+ & h\| \hat{w}_h \|_{1} \{ \| n w_\pi^f - n w^f \|_0 +  \| n w^f - \Pi_{l}^{0} (n w^f) \|_0 \} \\
& \lesssim & h^s \{ \| w^f \|_{1+s} + \| u^f \|_{1+s} \} \| \hat{w}_h \|_{1}.
\end{eqnarray*}
Additionally, because of
\begin{eqnarray*}
&& \text{Re} (a_h(({w}^f_I,{u}^f_I),(\hat{w}_h,\hat{u}_{h}))- a_h(({w}^f_h,{u}^f_h),(\hat{w}_h,\hat{u}_{h}))) \\
&& = \text{Re} \{ b((0,P_h f),(\hat{w}_h,\hat{u}_{h})) -b_h((0,f),(\hat{w}_h,\hat{u}_{h})) \}\\
&&~~~+ \text{Re} \{ a_h(({w}^f_I - w^f_\pi,{u}^f_I - u^f_\pi),(\hat{w}_h,\hat{u}_{h})) + a((w^f_\pi -{w}^f,u^f_\pi-{u}^f),(\hat{w}_h,\hat{u}_{h})) \}\\
&&~~~+  \text{Re} \{ a_h((w^f_\pi,u^f_\pi),(\hat{w}_h,\hat{u}_{h})) - a((w^f_\pi,u^f_\pi),(\hat{w}_h,\hat{u}_{h})) \} ,
\end{eqnarray*}
a similar analysis to $I^{\prime}$ and $I\!I^{\prime}$ yields
\begin{eqnarray*}
&& \text{Re} (a_h(({w}^f_I,{u}^f_I),(\hat{w}_h,\hat{u}_{h}))- a_h(({w}^f_h,{u}^f_h),(\hat{w}_h,\hat{u}_{h})))\\
&& \lesssim  h (\|f-f_I\|_0+\|f-\Pi_{l}^0f\|_0) \| \hat{u}_h \|_1 \\
&& \quad +h^s \{ \| w^f \|_{1+s} + \| u^f \|_{1+s} \}   \{ \| \hat{w}_h \|_{1} + \|\hat{u}_h \|_{1}\}
\end{eqnarray*}
thus,
\begin{eqnarray*}
 I\!I \lesssim  h^s \{ \| w^f \|_{1+s} + \| u^f \|_{1+s} \} +  h \|f-\Pi_{l}^{0}f\|_{0}+h\|f-f_I\|_0,
\end{eqnarray*}
which together with the estimate for $I$ yields \eqref{a3.20}.  \hfill $\Box$
\end{proof}
\indent Using Lemmas \ref{lem3.3}-\ref{lem3.4}, \eqref{a3.13} and \eqref{a3.14}, it is easy to show that $\| T_h - T \|_{L^2(\Omega) \rightarrow L^2(\Omega)}$ $ \to 0$. In fact, for $\forall f \in L^2(\Omega)$, we have
\begin{align*}
\| T_h - T \|_{L^2(\Omega) \rightarrow L^2(\Omega)} & = \sup_{0 \neq f \in  L^2(\Omega)} \frac{\|T_h f - Tf\|_{0}}{\| f \|_{0}}
=\sup_{0 \neq f \in  L^2(\Omega)} \frac{\|u^f_h - u^f\|_{0}}{\| f \|_{0}}\\ & \lesssim  h^{2r} \to 0 \text{ as } h \to 0.
\end{align*}
Using Lemma \ref{lem3.4}, we can similarly prove
$\| \mathcal{A}_h - \mathcal{A} \|_{\mathds{V} \rightarrow \mathds{V}} \to 0$. In fact, for $\forall {\bf F} \in \mathds{V}$, we have
\begin{align*}
\| \mathcal{A}_h - \mathcal{A} \|_{\mathds{V} \rightarrow \mathds{V}} &= \sup_{0 \neq {\bf F} \in  \mathds{V}} \frac{\|\mathcal{A}_h {\bf F} - \mathcal{A}{\bf F}\|_{\mathds{V}}}{\| {\bf F} \|_{\mathds{V}}}
= \sup_{0 \neq {\bf F} \in  \mathds{V}} \frac{\|\widetilde{\mathcal{A}}_h {\bf P}_h {\bf F} - \mathcal{A}{\bf F} \|_{\mathds{V}}}{\| {\bf F} \|_{\mathds{V}}} \\
&= \sup_{0 \neq {\bf F} \in  \mathds{V}} \frac{\| {\bf U}^{\bf F}_h - {\bf U}^{\bf F}\|_{\mathds{V}}}{\| {\bf F} \|_{\mathds{V}}} \lesssim  \sup_{0 \neq {\bf F} \in  \mathds{V}} \frac{h^{r}\|{\bf F}\|_\mathds{W}}{\|{\bf F}\|_\mathds{V}}
 \lesssim  h^{r} \to 0 \text{ as } h \to 0.
\end{align*}

Let $\lambda_j$ and $\lambda_{j,h}$ denote the $j$-th eigenvalue of \eqref{a2.3} and \eqref{a3.3}, respectively (the eigenvalues are sorted in ascending order based on their modulus, and when the moduli are equal, sorted in ascending order based on the size of their imaginary parts), and let the algebraic multiplicity of $\lambda_j$ be $q$, i.e., $\lambda_{j-1} \neq \lambda_{j} = \lambda_{j+1} = \cdots = \lambda_{j+q-1} \neq \lambda_{j+q}$.
Let $M(\lambda_j)$ be the space spanned by the eigenfunctions of \eqref{a2.3} corresponding to the eigenvalue $\lambda_j$,
and $M_h(\lambda_j)$ be the direct sum of the eigenspaces corresponding to the eigenvalues of \eqref{a3.3} that converge to $\lambda_j$.
Let $R(\lambda_j)$ be the space spanned by the eigenfunctions of $T$ corresponding to the eigenvalue $\lambda_j$,
and $R_{h}(\lambda_{j})$ be the direct sum of the eigenspaces corresponding to the eigenvalues of $T_{h}$ that converge to $\lambda_j$.
Obviously, if $(\lambda_j, {\bf U}_j=(w_j,u_j))$ is an eigenpair of $\eqref{a2.3}$, then $(\lambda_j, u_j)$ is an eigenpair of $T$;
The definitions of $M^*(\lambda_j^{*})$, $M_{h}^*(\lambda_j^{*})$, $R^*(\lambda_j^{*})$ and $R_{h}^*(\lambda_j^{*})$
are analogous to $M(\lambda_j)$, $M_{h}(\lambda_j)$, $R(\lambda_j)$  and $R_{h}(\lambda_{j})$, respectively.\\
\indent For two closed subspaces ${\bf \mathcal{M}}$ and ${\bf \mathcal{N}}$ of $\mathds{V}$, we recall the definition of the gap $ \widehat{\delta}_1$ between  ${\bf \mathcal{M}}$ and ${\bf \mathcal{N}}$ in the sense of norm
$\| \cdot \|_\mathds{V}$, that is
$$ \delta_1({\bf \mathcal{M}},{\bf \mathcal{N}})= \sup_{{\pmb \Phi} \in {\bf \mathcal{M}}, \|{\pmb \Phi}\|_\mathds{V}=1}
\inf_{{\pmb \Psi} \in {\bf \mathcal{N}}} \| {\pmb \Phi} - {\pmb \Psi} \|_\mathds{V},~
\widehat{\delta_1}({\bf \mathcal{M}},{\bf \mathcal{N}})= \max(\delta_1({\bf \mathcal{M}},{\bf \mathcal{N}}),\delta_1({\bf \mathcal{N}},{\bf \mathcal{M}})).
$$
\noindent Similarly, we can define the gap $\widehat{\delta}_0$ between two subspaces $\mathbf{\mathcal{M}}$ and $\mathbf{\mathcal{N}}$ of $L^2(\Omega)$ in the sense of the norm $\|\cdot\|_0$.\\
\indent Denote $ \widehat{\lambda}_{j,h}=\frac{1}{q} \sum_{i=j}^{j+q-1} \lambda_{i,h}$. Since $\| T_h - T \|_{L^2(\Omega) \rightarrow L^2(\Omega)} \to 0$, $\| \mathcal{A}_h - \mathcal{A} \|_{\mathds{V} \rightarrow \mathds{V}} \to 0$, using Lemmas \ref{lem3.3} and \ref{lem3.4}
we obtain the following a priori error estimate for eigenpairs according to the spectral approximation theory (see \cite{Ba1991}).
\begin{theorem} \label{th3.1}
Let $\lambda_j$ and $\lambda_{j,h}$ be the $j$-th eigenvalue of \eqref{a2.3} and \eqref{a3.3}, respectively, and suppose that $M(\lambda_j) \subset H^{1+s}(\Omega) \times H^{1+s}(\Omega) (r \leq s \leq l)$. 
Then, when $n$ is a piecewise smooth function, there hold
\begin{align}
\widehat{\delta}_1(M(\lambda_j),M_h(\lambda_j)) & \lesssim h^s, \label{a3.21}\\
\widehat{\delta}_0(R(\lambda_j),R_h(\lambda_j))& \lesssim h^{r +s}, \label{a3.22}\\
\mid \lambda_j - \widehat{\lambda}_{j,h} \mid &\lesssim h^{2 s}. \label{a3.23}
\end{align}
\end{theorem}
\begin{proof}
Since $\| \mathcal{A}_h - \mathcal{A} \|_{\mathds{V} \rightarrow \mathds{V}} \to 0$  and $\| T_h - T \|_{L^2(\Omega) \rightarrow L^2(\Omega)} \to 0$ as $h\to 0$, from Theorems 7.1 in \cite{Ba1991} we have
\begin{align}\label{a3.24}
\widehat{\delta}_1(M(\lambda_j),M_h(\lambda_j)) & \lesssim \|(\mathcal{A} - \mathcal{A}_h )\mid_{M(\lambda_{j})}\|_{\mathds{V}},\\
\widehat{\delta}_0(R(\lambda_j),R_h(\lambda_j))& \lesssim \|(T-T_{h})\mid_{R(\lambda_{j})}\|_{0}. \label{a3.25}
\end{align}
From Lemma \ref{lem3.4}, we obtain
\begin{eqnarray}
\| \mathcal{A} {\bf U} - \mathcal{A}_h {\bf U} \|_{\mathds{V}} \lesssim h^s, \quad \forall {\bf U}\in M(\lambda_{j}),\label{a3.26}
\end{eqnarray}
and from Lemma \ref{lem3.3}, we obtain
\begin{eqnarray}
\| T u - T_h u \|_{0}\lesssim h^{r+ s},\quad \forall u\in R(\lambda_{j}) . \label{a3.27}
\end{eqnarray}
Then, from (\ref{a3.24}) and (\ref{a3.26}) we get (\ref{a3.21}), and from (\ref{a3.25}) and (\ref{a3.27}) we get (\ref{a3.22}).\\
\indent Let $\xi_j,\cdots,\xi_{j+q-1}$ be a basis in $R(\lambda_j)$ and $\xi_j^*,\cdots,\xi_{j+q-1}^*$ be the
dual basis in $R^*(\lambda_j)$. Since $\| T_h - T \|_{L^2(\Omega) \rightarrow L^2(\Omega)} \to 0$, then from Theorems 7.2 in \cite{Ba1991} we have
\begin{align}
&\mid\lambda_j -\widehat{\lambda}_{j,h}\mid \nonumber \\
&\lesssim \frac{1}{q} \sum_{i=j}^{j+q-1}\mid b((0,(T-T_h)\xi_i),(0,\xi^*_i))\mid+ \| (T_h - T)\mid_{R(\lambda_j)} \|_{0} \| (T_h^* - T^*)\mid_{R^*(\lambda_j)} \|_{0}. \label{a3.28}
\end{align}
In \eqref{a3.28}, we recognize that the second term on the right-hand side is of higher order compared to the first term; therefore, we only need to estimate the first term on the right-hand side. For $\forall f \in R(\lambda_j), f^* \in R^*(\lambda_j)$, we have
\begin{align*}
&b((0,(T-T_h)f),(0,f^*)) = b((0,f),(0,(T^*-T_h^*)f^*))  \\
&=b({\bf F},(\mathcal{A}^*-\mathcal{A}_h^*){\bf F^*})=a(\mathcal{A}{\bf F},(\mathcal{A}^*-\mathcal{A}_h^*){\bf F^*})\\
&= a((\mathcal{A}-\mathcal{A}_h){\bf F},(\mathcal{A}^*-\mathcal{A}_h^*){\bf F^*})+\{b(\mathcal{A}_h{\bf F},{\bf F^*}) -b_h(\mathcal{A}_h{\bf F},{\bf P}_h{\bf F}^*  )\}\\
&\quad+ \{ a_h(\mathcal{A}_h{\bf F},\mathcal{A}_h^*{\bf F^*})- a(\mathcal{A}_h{\bf F},\mathcal{A}_h^*{\bf F^*})\}\\
&:= I + I\!I + I\!I\!I.
\end{align*}
\noindent From the continuity of $a(\cdot,\cdot)$ and the error estimate in the sense of $\mathds{V}$ norm for the source problem
we know that $I  \lesssim h^{2 s}$; using the argument for \eqref{a3.18} we have $I\!I \lesssim h^{2(r+s)}$; and using the argument for \eqref{a3.19} we also get $ I\!I\!I  \lesssim h^{2 s}$. So we arrive at $I + I\!I + I\!I\!I \lesssim h^{2s}$, which together with \eqref{a3.28} yields (\ref{a3.23}). \hfill $\Box$
\end{proof}
\indent In addition, assuming that the ascent of the eigenvalue $\lambda_j$ is 1, we also have the following estimates for the eigenfunction ${\mathbf U}_j=(w_j,u_j)$.
\begin{theorem} 
Suppose $n$ is a piecewise smooth function, and let $\lambda_j$ be the $j$-th eigenvalue of \eqref{a2.3} with the ascent of 1. Additionally, let $(\lambda_{j,h}, {\mathbf U}_{j,h})$ be the $j$-th eigenpair of \eqref{a3.3}. Suppose that  $M(\lambda_j) \subset H^{1+s}(\Omega) \times H^{1+s}(\Omega)~(r\leq s \leq l)$, then there is ${\mathbf U}_j\in M(\lambda_j)$  such that for sufficiently small $h$, the following estimates hold: 
\begin{align}
\| {\mathbf U}_j - {\mathbf U}_{j,h} \|_{\mathbf{\mathbb{V}}}  &\lesssim  h^s , \label{a3.29} \\
 \| u_j - u_{j,h} \|_{0} &\lesssim  h^{r+s} , \label{a3.30}\\
 \| u_j - u_{j,h} \|_{0} \lesssim  h^r \{\|{\mathbf U}_j-{\mathbf U}_{j,h}\|_{\mathbb{V}} &+ \| w_j-w_{j,\pi} \|_{1,h} + \| u_j-u_{j,\pi} \|_{1,h}  \} +  h^{1+r} \nonumber \\
\{ \|u_j-u_{j,I}\|_0+ \|u_j-\Pi_l^0 u_j\|_0 + &\|w_j-\Pi_{l}^{0}w_j\|_{0} + \|n w_j-\Pi_{l}^{0}(nw_j)\|_{0}  \} , \label{a3.31}
\end{align}
where $w_{j,\pi}$ and $u_{j,\pi}$ are the approximations of $w_j$ and $u_j$ that satisfy \eqref{a3.13}, respectively, and $u_{j,I}$ is the interpolation of $u_j$ satisfying \eqref{a3.14}.
\end{theorem}
\begin{proof}
Since $\| \mathbf{\mathcal{A}}_h - \mathbf{\mathcal{A}} \|_{\mathbf{\mathbb{V}} \to \mathbf{\mathbb{V}}} \to 0$ and $\| T_h - T \|_{L^2(\Omega) \to L^2(\Omega)} \to 0$, by Theorem 7.4 in \cite{Ba1991}, we know that there exists ${\mathbf U}_j \in M(\lambda_j)$, such that
\begin{eqnarray}
\| {\mathbf U}_j - {\mathbf U}_{j,h} \|_{\mathbf{\mathbb{V}}} &\lesssim& \| (\mathbf{\mathcal{A}}_h-\mathbf{\mathcal{A}}){{\mathbf U}_j} \|_{\mathbf{\mathbb{V}}}, \label{a3.32}\\
\|u_j-u_{j,h}\|_{0} &\lesssim& \| (T_h -T)u_j \|_{0}.\label{a3.33}
\end{eqnarray}
By combining \eqref{a3.32} and \eqref{a3.26}, we obtain \eqref{a3.29}. Similarly, combining \eqref{a3.33} and \eqref{a3.27}, we can derive \eqref{a3.30}. \\
\indent From the definition of $\mathbf{\mathcal{A}}$ and \eqref{a2.3} we get
$\mathbf{\mathcal{A}} {\mathbf U}_j=\frac{1}{\beta-\lambda_j} {\mathbf U}_j$. Similarly, there holds $\mathbf{\mathcal{A}}_h {\mathbf U}_{j,h}=\frac{1}{\beta-\lambda_{j,h}} {\mathbf U}_{j,h}$.
By using \eqref{a3.16}, \eqref{a3.33} and $\| (T_h -T)u_j \|_{0} \lesssim \| (\mathbf{\mathcal{A}}_h-\mathbf{\mathcal{A}}){{\mathbf U}_j} \|_{\mathbf{\mathbb{W}}}$, we get
\begin{align}
&\| u_j - u_{j,h} \|_{0} \lesssim  h^r \left\{ \|(\mathbf{\mathcal{A}}_h-\mathbf{\mathcal{A}}){{\mathbf U}_j}\|_{\mathbf{\mathbb{V}}} + \| w_j-w_{j,\pi} \|_{1,h} + \| u_j-u_{j,\pi} \|_{1,h} \right\} +  h^{1+r}\nonumber \\
&\quad \left\{\|u_j-u_{j,I}\|_0+ \|u_j-\Pi_{l}^{0}u_j\|_{0} + \|w_j-\Pi_{l}^{0}w_j\|_{0} + \|n w_j-\Pi_{l}^{0}(nw_j)\|_{0}  \right\}, \label{a3.34}
\end{align}
where $w_{j,\pi}$ and $u_{j,\pi}$ are the approximations of $w_j$ and $u_j$ that satisfy \eqref{a3.13}, respectively, and $u_{j,I}$ is the interpolation of $u_j$ satisfying \eqref{a3.14}.
And
\begin{eqnarray}
\|(\mathbf{\mathcal{A}}_h-\mathbf{\mathcal{A}}){\mathbf U}_j\|_{\mathbf{\mathbb{V}}}
&=&\|\mathbf{\mathcal{A}}_h {\mathbf U}_j -\mathbf{\mathcal{A}}_h {\mathbf U}_{j,h}+ \mathbf{\mathcal{A}}_h {\mathbf U}_{j,h}-\mathbf{\mathcal{A}} {\mathbf U}_j\|_{\mathbf{\mathbb{V}}} \nonumber\\
&\leq& \|\mathbf{\mathcal{A}}_h {\mathbf U}_j -\mathbf{\mathcal{A}}_h {\mathbf U}_{j,h}\|_{\mathbf{\mathbb{V}}}
+ \|\mathbf{\mathcal{A}}_h {\mathbf U}_{j,h}-\mathbf{\mathcal{A}} {\mathbf U}_j\|_{\mathbf{\mathbb{V}}} \nonumber\\
&\leq & \|\mathbf{\mathcal{A}}_h\|_{\mathbf{\mathbb{V}}} \|{\mathbf U}_j - {\mathbf U}_{j,h}\|_{\mathbf{\mathbb{V}}} \nonumber\\
&\quad +& \frac{1}{(\beta-\lambda_{j,h})(\beta-\lambda_j)}\|(\beta-\lambda_j){\mathbf U}_{j,h}-(\beta-\lambda_{j,h}){\mathbf U}_j\|_{\mathbf{\mathbb{V}}} \nonumber\\
&\lesssim &  \|{\mathbf U}_j - {\mathbf U}_{j,h}\|_{\mathbf{\mathbb{V}}} + \mid \lambda_{j,h}-\lambda_j\mid \nonumber\\
&\lesssim &  \|{\mathbf U}_j - {\mathbf U}_{j,h}\|_{\mathbf{\mathbb{V}}}, \label{a3.35}
\end{eqnarray}
where the last "$\lesssim$" in the above holds because, from Theorem 7.3 in \cite{Ba1991}, the proof of \eqref{a3.28} and \eqref{a3.29}, we know that $\mid\lambda_{j,h}-\lambda_j\mid$ is a small quantity compared to $\|{\mathbf U}_j - {\mathbf U}_{j,h}\|_{\mathbf{\mathbb{V}}}$. By combining \eqref{a3.34} and \eqref{a3.35}, we obtain \eqref{a3.31}.  
\hfill $\Box$
\end{proof}

\section{Numerical Experiments}
\indent To validate the theoretical analysis, in this section we will report some numerical examples. The computations are performed using MATLAB 
on a ThinkBook 14p Gen 2 PC with 16GB of RAM. In the numerical experiments, unless otherwise specified, the stabilization parameters $\sigma^E$ and $\tau^E$ are both set to $10^0$. We consider the test domains as $\Omega_S = [0, 1]^2$, $\Omega_L = [0, 1]^2 \setminus (0.5, 1) \times (0, 0.5)$, and the unit disk $\Omega_C$.
Referring to \cite{Liu2023,CoM2017}, the parameters are set to the following three cases:
\begin{description}
\setlength{\itemsep}{0pt}
\setlength{\parsep}{0pt}
\setlength{\parskip}{0pt}
\item[\textnormal{Case 1}:] $n=4,~\gamma=2,~k=1$ ;
\item[\textnormal{Case 2}:] $n=2+4\mathrm{i},~\gamma=2,~k=1$.
\item[\textnormal{Case 3}:] $n=4+2(x_1^2+x_2^2)^{1/2},~\gamma=2,~k=1$.
\end{description}
We use uniform refinement during the mesh refinement process, where for a geometric element $E$ with $E_N$ edges, we connect the midpoint of each edge of $E$ to the centroid of $E$, dividing $E$ into a mesh of no more than $E_N$ geometric elements (see, e.g., \cite{MoraR2020}).\\
\indent 
In order to rewrite the discrete variational form \eqref{a3.3} as a generalized matrix eigenvalue problem, we first let
a basis $\{\xi_i\}_{i=1}^{N_0+N_b}$ of the virtual element space $\mathcal{W}^l_h$, where $\{\xi_i\}_{i=N_0+1}^{N_0+N_b}$ correspond to the global degrees of freedom associated with $(Dof_1)$ and $(Dof_2)$ on $\partial \Omega$, and $\{\xi_i\}_{i=1}^{N_0}$ are the remaining internal degrees of freedom.\\
\indent Next, we define three $(N_0+N_b)\times (N_0+N_b)$ matrices $A_1$, $M_1$, and $M_2$, whose entries in the $j$-th row and $i$-th column are given respectively by
$$
(A_1)_{ji}=\sum_{E \in \mathcal{T}_h} a_h^E(\xi_i, \xi_j),~
(M_1)_{ji}=\sum_{E \in \mathcal{T}_h} k^2 b_h^E(n\xi_i, \xi_j),~
(M_2)_{ji}=\sum_{E \in \mathcal{T}_h} k^2 b_h^E(\xi_i, \xi_j).
$$
Let
$$
 w_h=\sum_{i=1}^{N_0+N_b} w_i \xi_i,\quad u_h=\sum_{i=1}^{N_0+N_b} u_i \xi_i,
$$
and set $w_i = u_i$ for $N_0+1 \leq i \leq N_0+N_b$ to ensure the condition $(w_h - u_h)\!\mid_e = 0$ for all $e \in \mathcal{E}_h \cap \partial\Omega$.
Then define
$$
X_h = \sum_{i=1}^{N_0} w_i \xi_i + \sum_{i=1}^{N_0} u_i \xi_i + \sum_{i=N_0+1}^{N_0+N_b} w_i \xi_i
$$
with ${\mathbf X}=(w_1,\cdots,w_{N_0},u_1,\cdots,u_{N_0},w_{N_0+1},\cdots,w_{N_0+N_b})^T$.
We introduce the matrices
$A =$
$$
\begin{bmatrix}
[A_1-M_1]_{N_0 \times N_0} & {\bf 0} & [A_1 - M_1]_{N_0 \times N_b} \\[4pt]
{\bf 0} & [\beta M_2-\frac{1}{\gamma}A_1]_{N_0 \times N_0} & [\beta M_2-\frac{1}{\gamma}A_1 ]_{N_0 \times N_b} \\[4pt]
[A_1 - M_1]_{N_b \times N_0} & [\beta M_2-\frac{1}{\gamma}A_1 ]_{N_b \times N_0} & [(1 - \frac{1}{\gamma})A_1 + (\beta M_2 - M_1)]_{N_b \times N_b}
\end{bmatrix},
$$
and
$$
M =
\begin{bmatrix}
{\bf 0} & {\bf 0} & {\bf 0} \\[4pt]
{\bf 0} & [M_2]_{N_0 \times N_0} & [M_2]_{N_0 \times N_b} \\[4pt]
{\bf 0} & [M_2]_{N_b \times N_0} & [M_2]_{N_b \times N_b}
\end{bmatrix}.
$$
Then the discrete formulation \eqref{a3.3} is transformed into the following generalized matrix eigenvalue problem:
$$
A {\mathbf X} = (\beta - \lambda_h) M {\mathbf X}.
$$
\noindent \textbf{Example 1:} For the test domain $\Omega_S$, we consider three different polygonal meshes, with the initial meshes shown in Fig.~\ref{fig:1}.

\begin{figure}[ht]
  \captionsetup[subfloat]{labelformat=empty} 
    \centering
    \subfloat[quadrilateral mesh $\mathcal{T}_h^1$]{\label{fig:1a}\includegraphics[width=.33\linewidth]{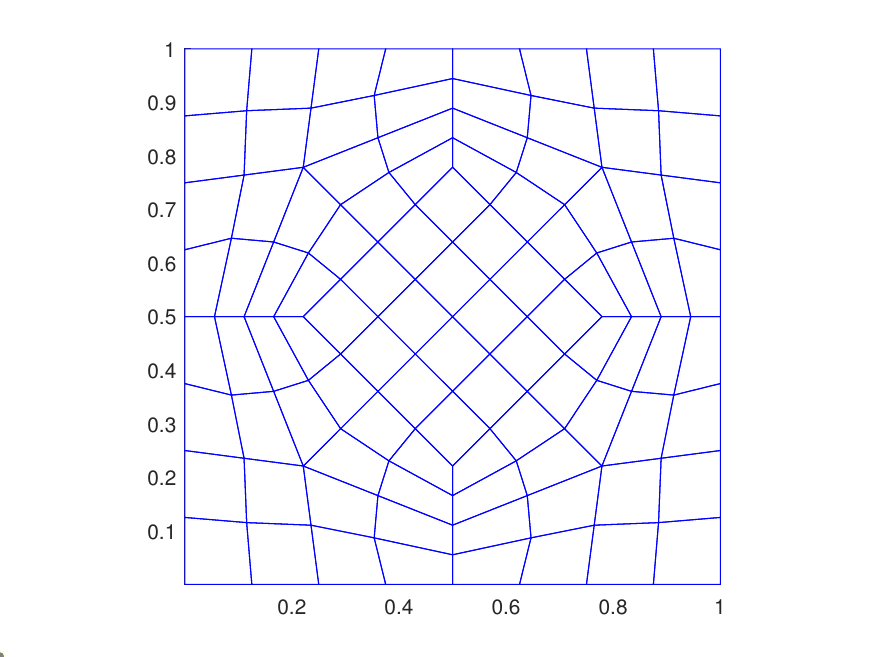}}
    \subfloat[Voronoi mesh $\mathcal{T}_h^2$]{\label{fig:1b}\includegraphics[width=.33\linewidth]{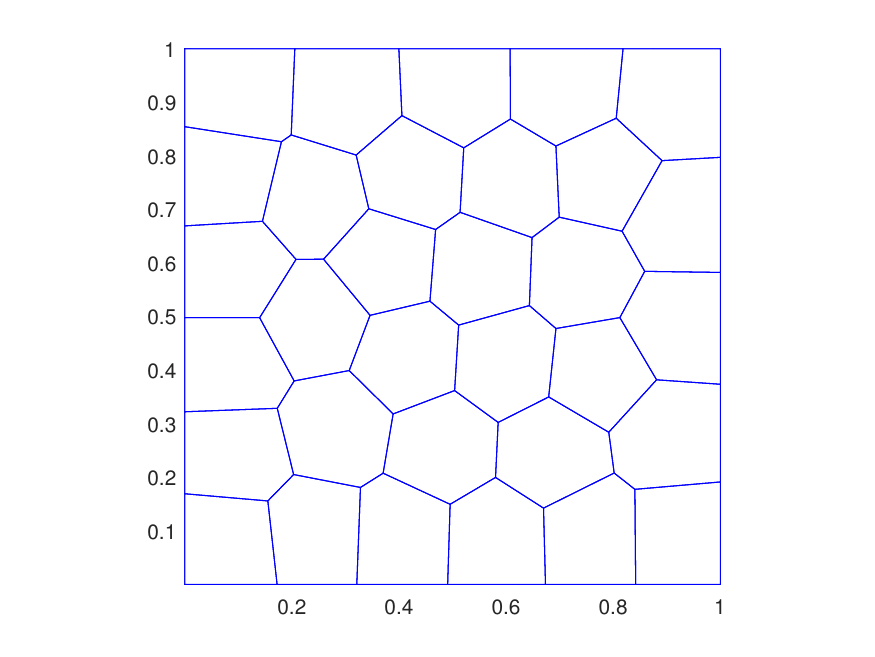}}
    \subfloat[distortion poly mesh $\mathcal{T}_h^3$]{\label{fig:1c}\includegraphics[width=.33\linewidth]{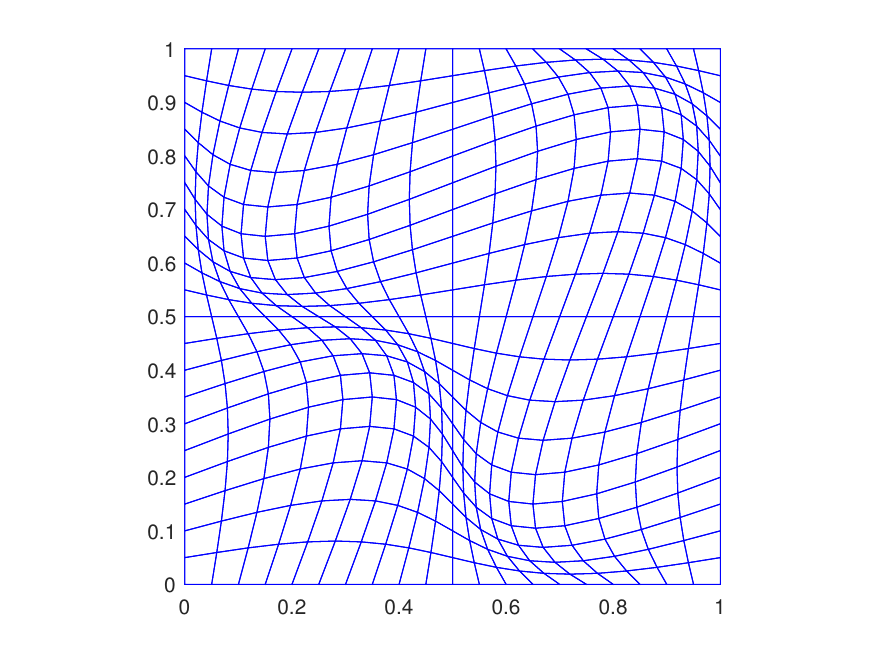}}
    \caption{Initial meshes for $\Omega_S$}
    \label{fig:1}
\end{figure}

\begin{figure}[ht]
    \centering
    \subfloat[]{\label{fig1a}\includegraphics[width=.33\linewidth]{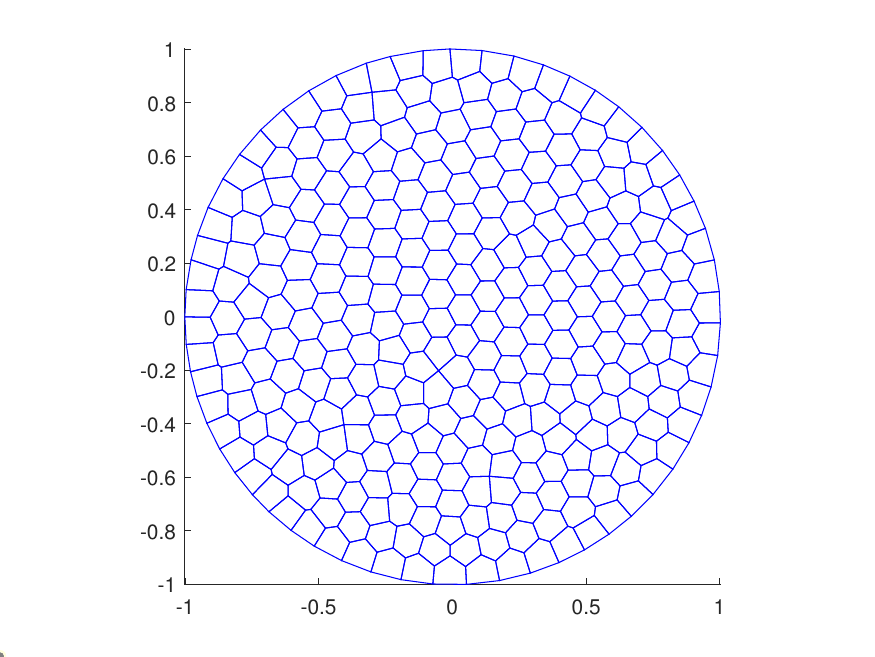}}
    \subfloat[]{\label{fig1a}\includegraphics[width=.33\linewidth]{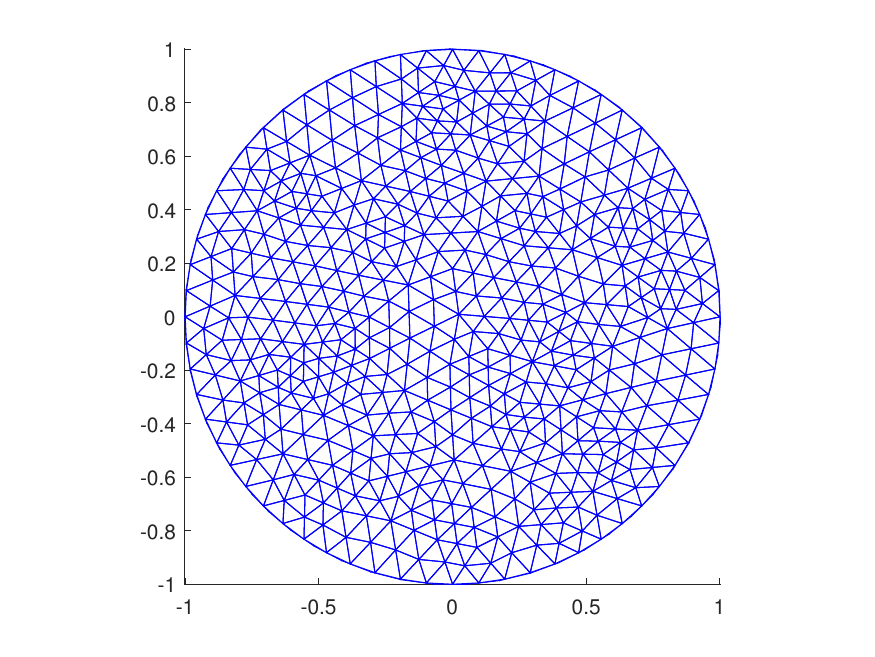}}
    \subfloat[]{\label{fig1b}\includegraphics[width=.33\linewidth]{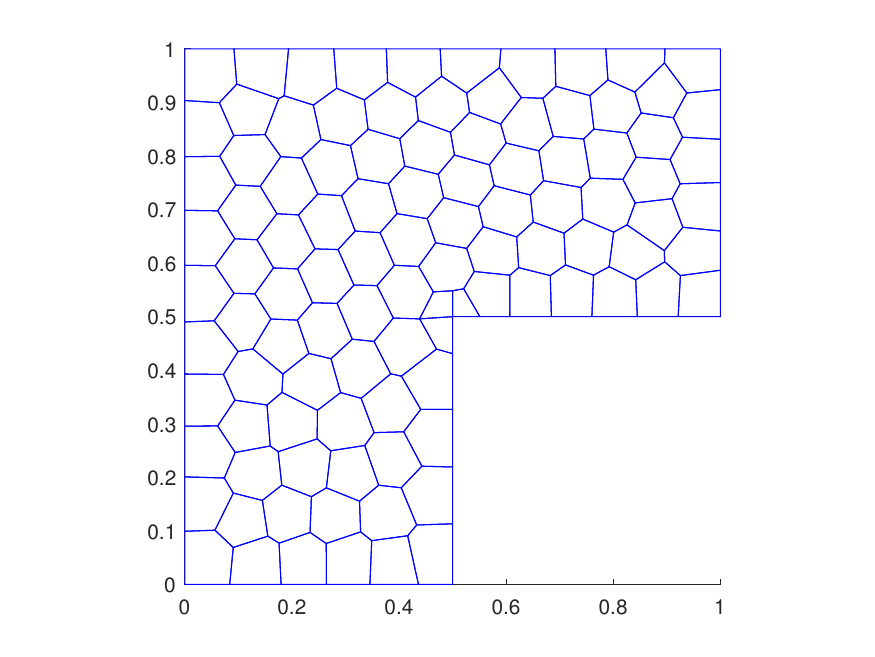}}
    \caption{Two different mesh partitions of $\Omega_C$: (a) and (b), and initial mesh for $\Omega_L$: (c)}
    \label{fig:2}
\end{figure}

\indent In Case 1, we compute the first four approximate eigenvalues of \eqref{a2.3} on $\Omega_S$ using the VEM with degrees $l=1$ and $l=2$. The corresponding error curves for these approximate eigenvalues are plotted in Figs.~\ref{fig:3} to~\ref{fig:5} (where the reference values are taken as the most accurate approximations that we can compute). From Figs.~\ref{fig:3} to~\ref{fig:5}, we observe that, with the exception of some fluctuations in the third approximate eigenvalue on the right side of Fig.~\ref{fig:4}, the error curves are generally parallel to the lines with slopes of $-1$ or $-2$, indicating that our method achieves the optimal convergence rate. Additionally, we can clearly see that the convergence results with the mesh partition $\mathcal{T}_h^3$ are better than those with $\mathcal{T}_h^1$ and $\mathcal{T}_h^2$.

\begin{figure}[ht]
\captionsetup[subfloat]{labelformat=empty} 
    \centering
    \subfloat[Study of convergence when $ l = 1 $]{\label{fig1a}\includegraphics[width=.45\linewidth]{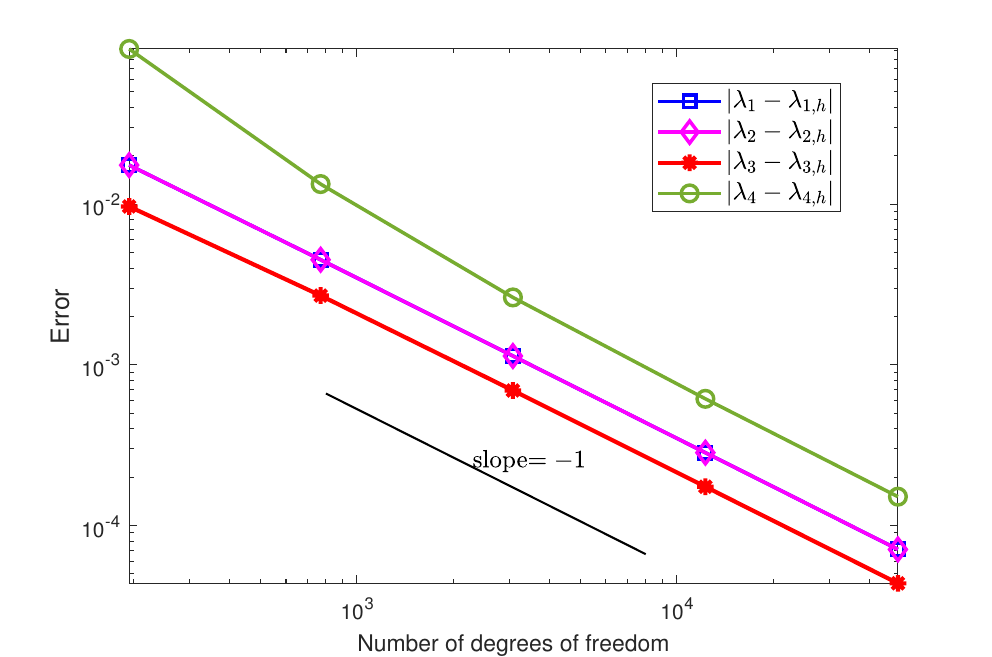}}
    \subfloat[Study of convergence when $ l = 2 $]{\label{fig1b}\includegraphics[width=.45\linewidth]{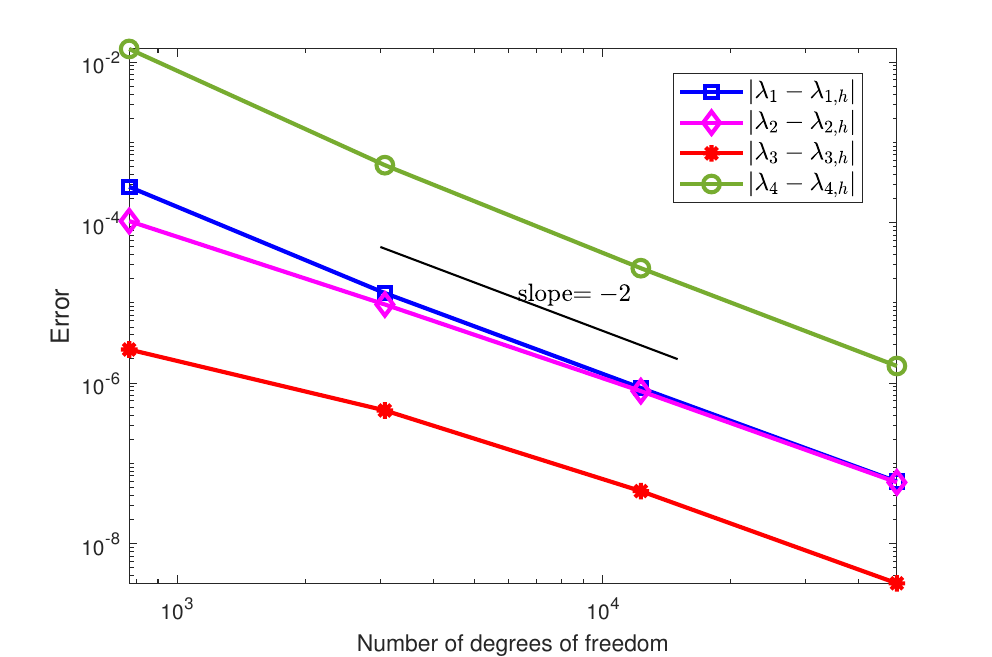}}
    \caption{Error curves of the first four approximate eigenvalues on $ \mathcal{T}_h^1 $ in Case 1}
    \label{fig:3}
\end{figure}

\begin{figure}[ht]
\captionsetup[subfloat]{labelformat=empty} 
    \centering
    \subfloat[Study of convergence when $ l = 1 $]{\label{fig1a}\includegraphics[width=.45\linewidth]{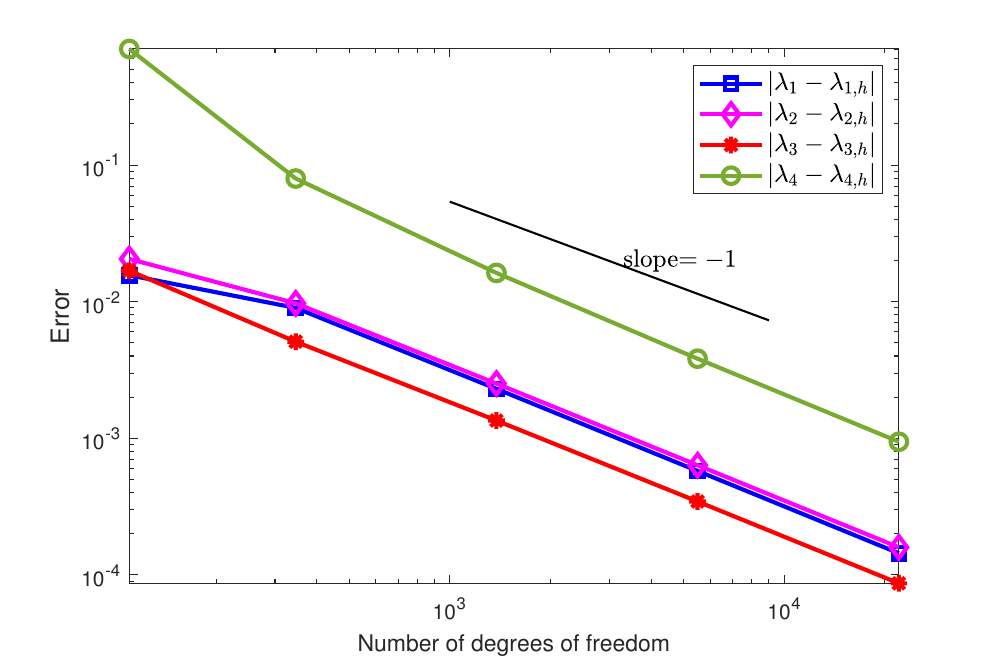}}
    \subfloat[Study of convergence when $ l = 2 $]{\label{fig1b}\includegraphics[width=.45\linewidth]{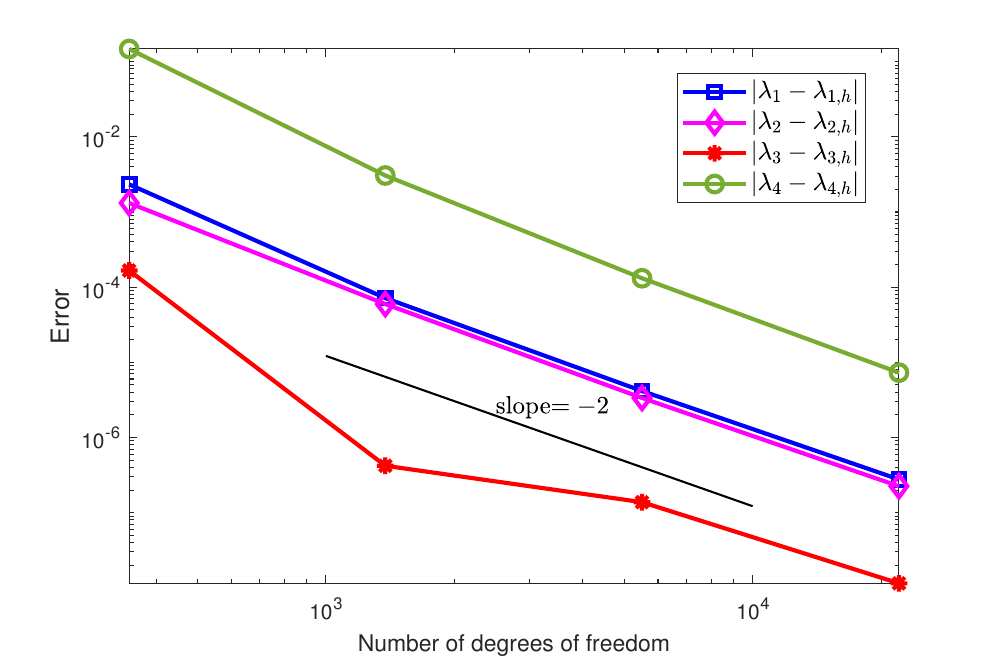}}
    \caption{Error curves of the first four approximate eigenvalues on $\mathcal{T}_h^2$ in Case 1}
    \label{fig:4}
\end{figure}

\begin{figure}[ht]
\captionsetup[subfloat]{labelformat=empty} 
    \centering
    \subfloat[Study of convergence when $ l = 1 $]{\label{fig1a}\includegraphics[width=.45\linewidth]{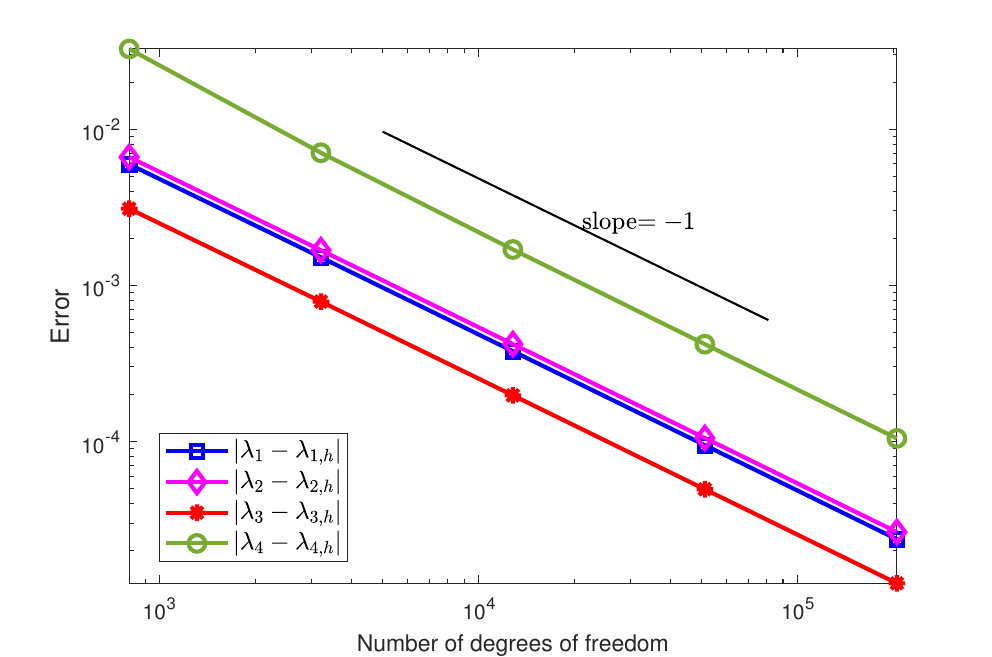}}
    \subfloat[Study of convergence when $ l = 2 $]{\label{fig1b}\includegraphics[width=.45\linewidth]{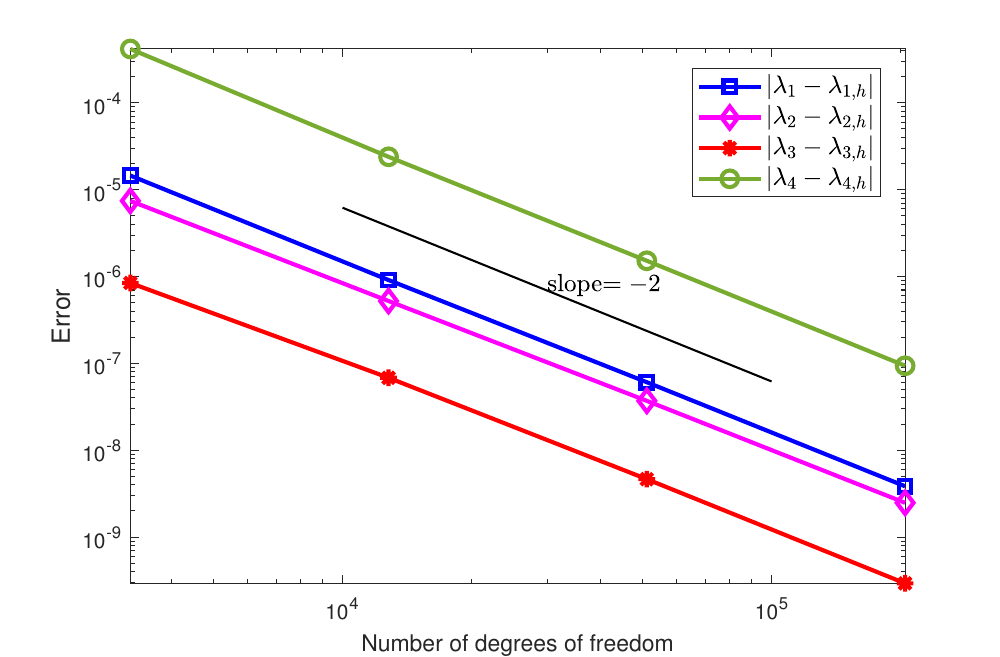}}
    \caption{Error curves of the first four approximate eigenvalues on $\mathcal{T}_h^3$ in Case 1}
    \label{fig:5}
\end{figure}

\noindent {\bf Example 2:}
In Case 1, we compute the first four approximate eigenvalues of \eqref{a2.3} on the domain $\Omega_L$ using the VEM with degrees $l=1$ and $l=2$. The initial mesh used is depicted in Fig.~\ref{fig:2}(c). The reference values are selected using the same method as in Example 1. 
Fig.~\ref{fig:6} shows the corresponding error curves. For $l = 1$, all eigenvalues essentially achieve the optimal convergence order. For $l = 2$, these error curves are not parallel to the line with slope $-2$, nor do they achieve parallelism with the line of slope $-1$, which is consistent with the theory, since the convergence order $s$ depends on both $l$ and $r$.
Subsequently, by setting the stabilization parameters $\sigma^E$ and $\tau^E$ to the respective values of $10^{-3}, 10^{-2}, 10^{-1}, 10^{0}, 10^{1}, 10^{2}$, we test their influence on the convergence of the first eigenvalue of \eqref{a3.3}. The convergence results are plotted in Figs.~\ref{fig:7} and~\ref{fig:8}. From these results, it is evident that the selection of the stabilization parameters $\sigma^E$ and $\tau^E$ is not trivial. To achieve good convergence, choosing either $\sigma^E=\tau^E=1$ or a larger value for $\sigma^E$ seems advisable. 
In particular, Fig.~\ref{fig:7} indicates that for $l=1$ and $\tau^E \in [10^{-3}, 10^{1}]$, the optimal choice appears to be $\sigma^E \in [10^{-1}, 10^{0}]$. Fig.~\ref{fig:8} shows that for $l=2$ and $\tau^E \in [10^{-3}, 10^{2}]$, the optimal choice is $\sigma^E = 10^{1}$.

\begin{figure}[ht]
\captionsetup[subfloat]{labelformat=empty} 
    \centering
    \subfloat[Study of convergence when $ l = 1 $]{\label{fig1a}\includegraphics[width=.45\linewidth]{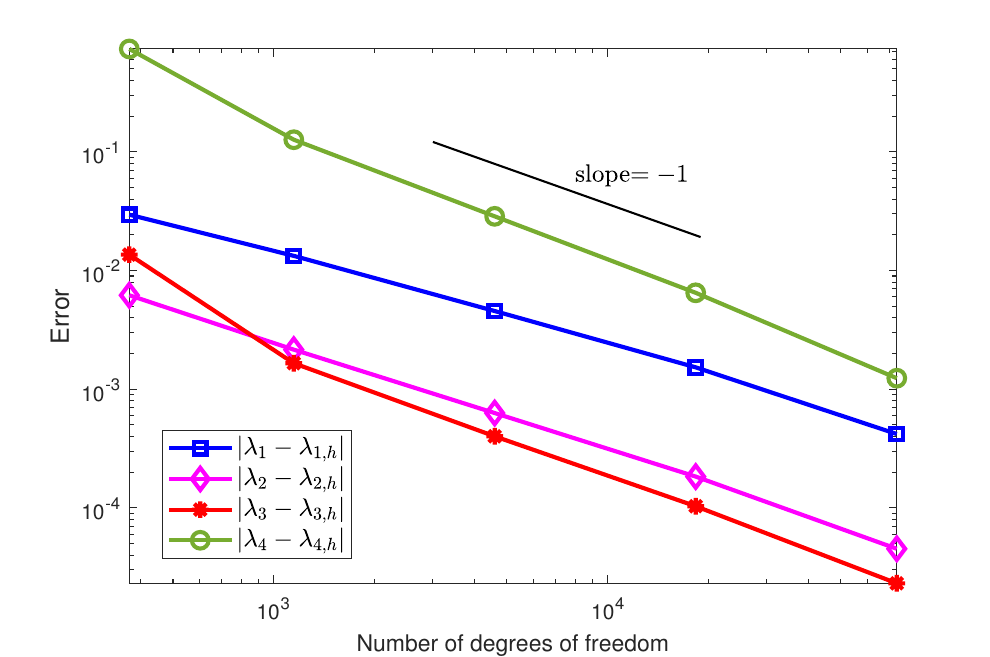}}
    \subfloat[Study of convergence when $ l = 2 $]{\label{fig1b}\includegraphics[width=.45\linewidth]{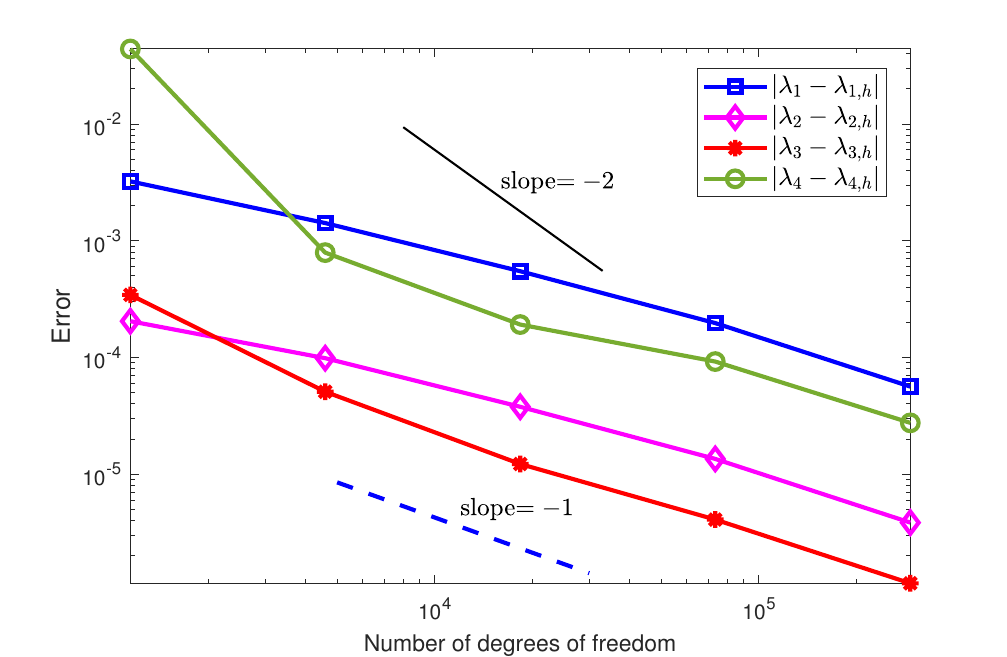}}
    \caption{Error curves of the first four approximate eigenvalues on $\Omega_L$ in Case 1}
    \label{fig:6}
\end{figure}

\begin{figure}[ht]
\flushleft
\includegraphics[width=1\textwidth,height=0.5\textwidth]{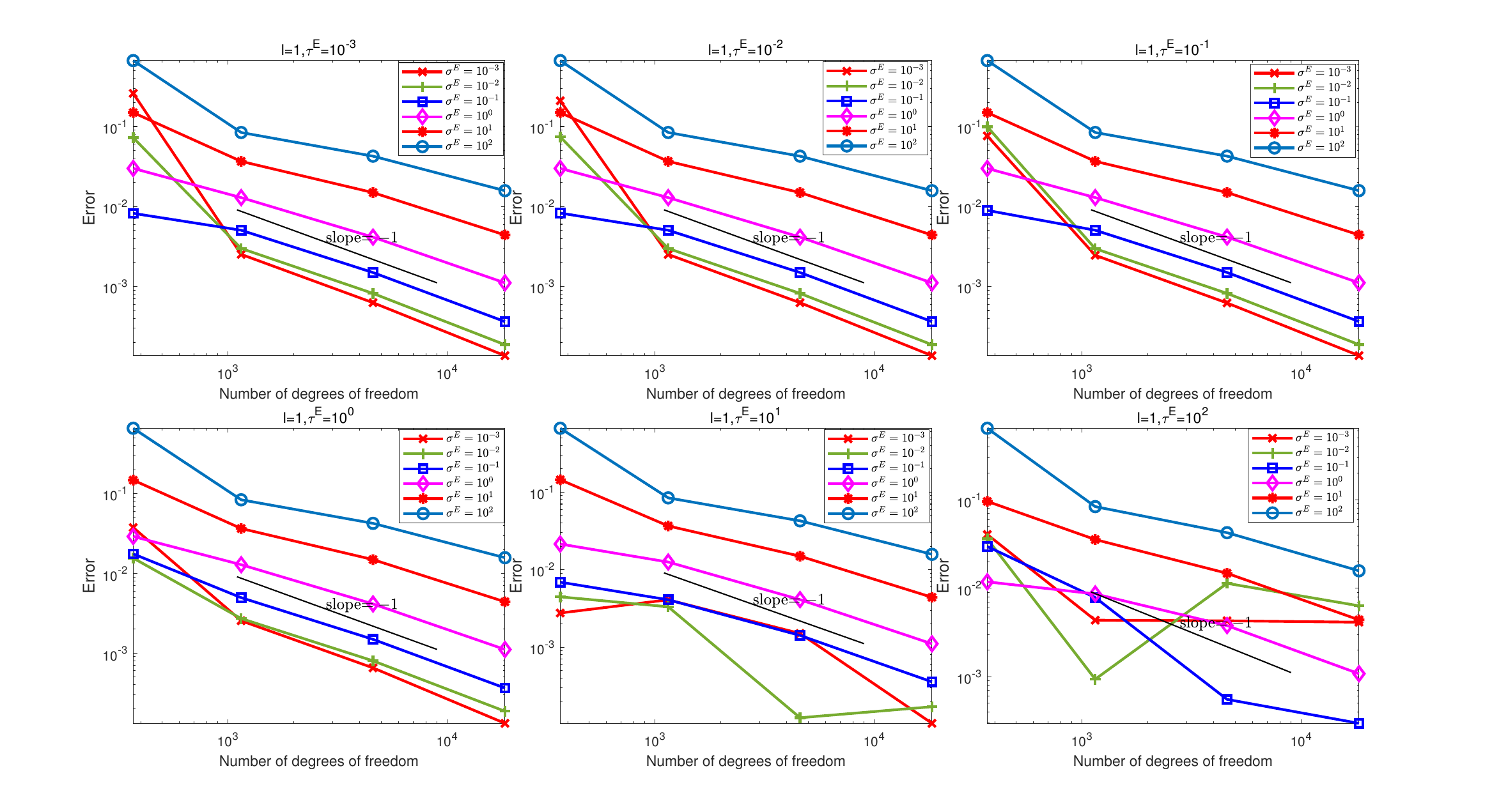}
\caption{The error curve of the first approximate eigenvalue obtained by using the VEM with degree $l=1$ on $\Omega_L$ with different parameter selections for $\sigma^E$ and $\tau^E$}
\label{fig:7}
\end{figure}

\begin{figure}[ht]
\centering
\includegraphics[width=1\textwidth,height=0.5\textwidth]{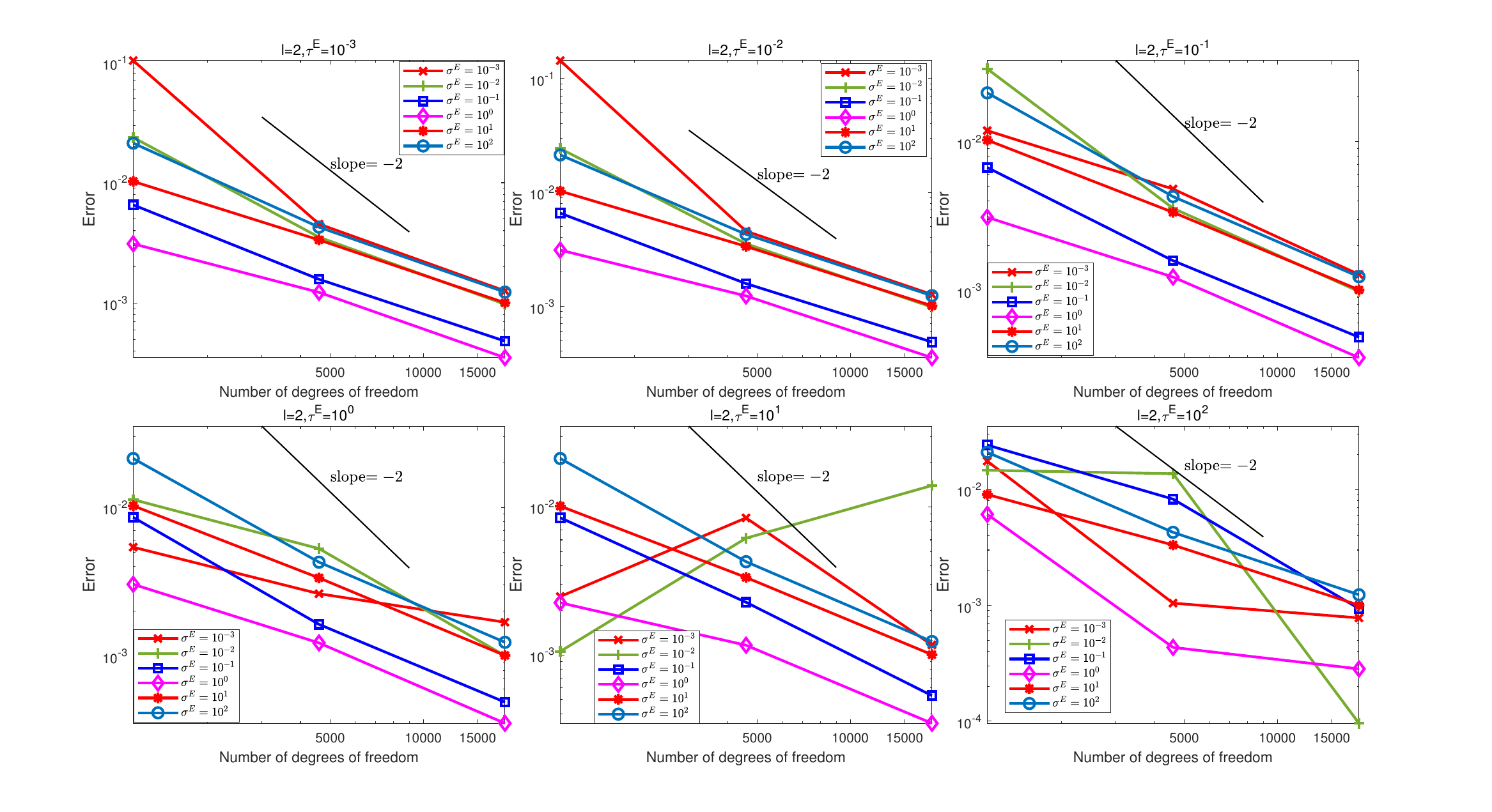}
\caption{The error curve of the first approximate eigenvalue obtained by using the VEM with degree $l=2$ on $\Omega_L$ with different parameter selections for $\sigma^E$ and $\tau^E$}
\label{fig:8}
\end{figure}

\noindent {\bf Example 3:}
We compute the first four approximate eigenvalues of \eqref{a2.3} on $\Omega_C$ using the VEM degrees $l=1$ and $l=2$. The mesh generation we employed is shown in 
Fig.~\ref{fig:2} $(a)$ and $(b)$. The error curves are plotted in Fig.~\ref{fig:9} and~\ref{fig:10}. To achieve this, and also to reduce geometric errors during the uniform refinement process, we uniformly adjust any newly generated boundary nodes that are not on the circle so that they lie exactly on the circle's boundary. The reference values are selected using the same method as in Example 1. Additionally, for the $(b)$-type mesh, we compute the first four approximate eigenvalues of \eqref{a2.3} using both VEM and FEM, with the numerical results presented in Table \ref{tab5}. For $l=1$, Fig.~\ref{fig:9} indicates that our method achieves optimal convergence order. For $l=2$, it can be observed from Fig.~\ref{fig:10} that the geometric error plays a non-negligible role.
From Table \ref{tab5}, we see that when using $l=1$, the degrees of freedom and the accuracy of the approximate eigenvalues for both VEM and FEM are essentially the same. However, when using $l=2$, due to the local degrees of freedom $(Dof_3)$ in VEM, it is evident that VEM has more degrees of freedom compared to FEM.
Combining the observations from Figs.~\ref{fig:3} to~\ref{fig:5} and Table \ref{tab5}, we find that VEM has an advantage over FEM when handling irregular meshes, but for regular triangular partitions, higher-order VEM does not have a significant advantage over the FEM. Nevertheless, on finer meshes, both VEM and FEM achieve at least 5 to 6 digits of accuracy in their approximations.

\begin{figure}[ht]
\captionsetup[subfloat]{labelformat=empty} 
    \centering
    \subfloat[$(a)$-type mesh]{\label{fig1a}\includegraphics[width=.45\linewidth]{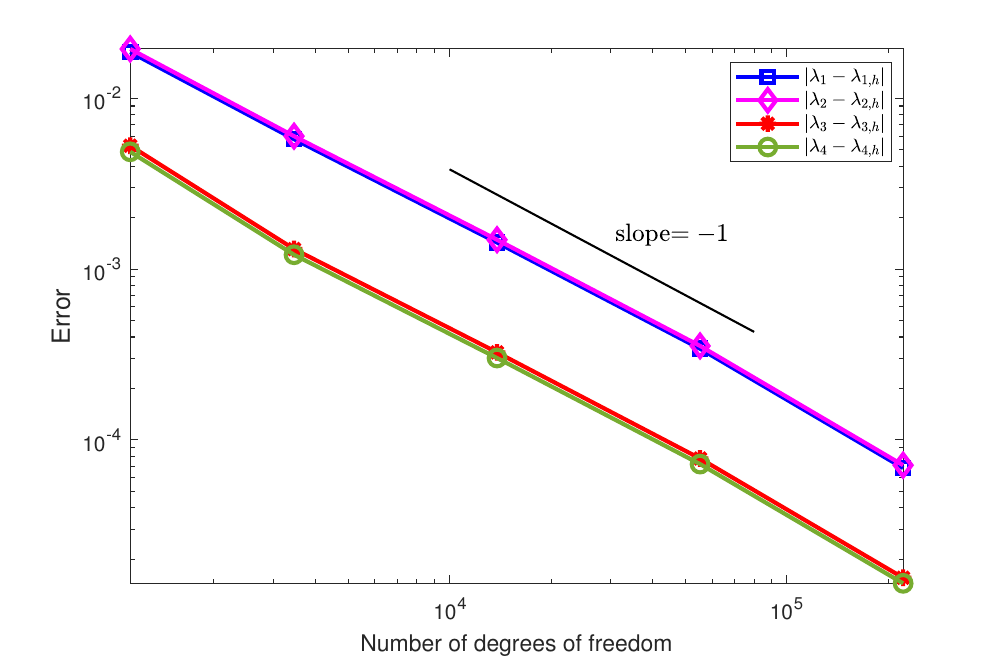}}
    \subfloat[$(b)$-type mesh]{\label{fig1b}\includegraphics[width=.45\linewidth]{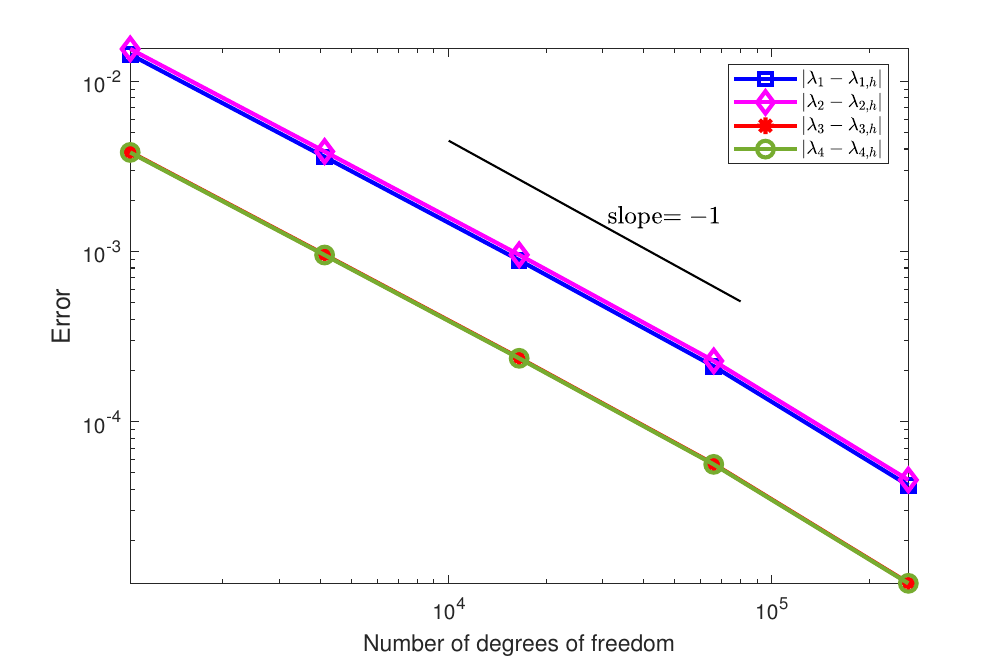}}
    \caption{In Case 1, error curves of the first four approximate eigenvalues obtained using the VEM with degree $l=1$ on $\Omega_C$}
    \label{fig:9}
\end{figure}

\begin{figure}[ht]
\captionsetup[subfloat]{labelformat=empty} 
    \centering
    \subfloat[$(a)$-type mesh]{\label{fig1a}\includegraphics[width=.45\linewidth]{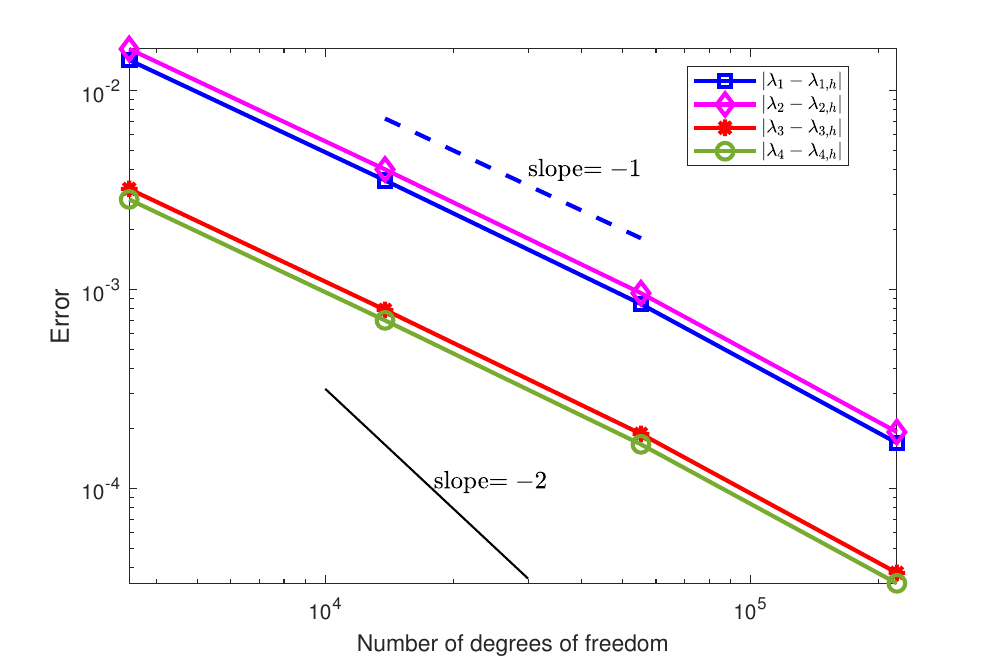}}
    \subfloat[$(b)$-type mesh]{\label{fig1b}\includegraphics[width=.45\linewidth]{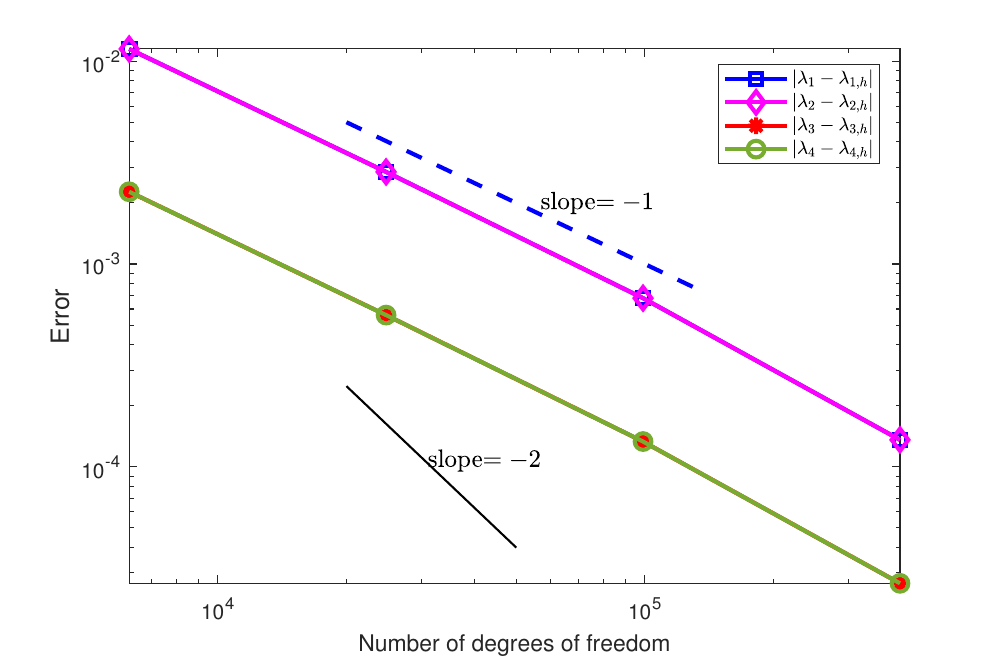}}
    \caption{In Case 1, error curves of the first four approximate eigenvalues obtained using the VEM with degree $l=2$ on $\Omega_C$}
    \label{fig:10}
\end{figure}

\begin{table}[ht]
\centering\begin{minipage}{\textwidth}
\caption{The first four approximate eigenvalues obtained by VEM and FEM on the $(b)$-type mesh generation of $\Omega_C$}
\label{tab5}
\begin{tabular*}{\textwidth}{@{\extracolsep{\fill}}*{7}{l}@{\extracolsep{\fill}}}
\toprule%
{$l$}  &methods  &$Ndof$    &$\lambda_{1,h}$ &$\lambda_{2,h}$  &$\lambda_{3,h}$ &$\lambda_{4,h}$ \\
\midrule
 \multirow{10}{*}{$1$} &\multirow{5}{*}{VEM}
      & 1034    & -1.74042528  & -1.74153847  & 2.27816716  & 2.27819096   \\
   &  & 4130    & -1.72964942  & -1.72992467  & 2.28105053  & 2.28105654    \\
   &  & 16514   & -1.72694329  & -1.72701182  & 2.28177276  & 2.28177427  \\
   &  & 66050   & -1.72626607  & -1.72628318  & 2.28195349  & 2.28195387   \\
   &  & 264194  & -1.72609678  & -1.72610106  & 2.28199870  & 2.28199879   \\
                                 \cmidrule{3-7}
                                 &\multirow{5}{*}{FEM}  & 1034 &-1.74042528  & -1.74153847  & 2.27816716  & 2.27819096  \\
                                & & 4130   &-1.72964942  & -1.72992467  & 2.28105053  & 2.28105654  \\
                                & & 16514  &-1.72694329  & -1.72701182  & 2.28177276  & 2.28177427  \\
                                & & 66050  & -1.72626607 & -1.72628318  & 2.28195349  & 2.28195387  \\
                                & & 264194 &-1.72609678  & -1.72610106  & 2.28199870  & 2.28199879  \\ 
                                \cmidrule{2-7} 
  \multirow{8}{*}{$2$} &\multirow{4}{*}{VEM}
      & 6194   &-1.73757028  & -1.73757194  & 2.27973410  & 2.27973412   \\
    & & 24770  &-1.72893022  & -1.72893024  & 2.28144483  & 2.28144483   \\
    & & 99074  &-1.72676448  & -1.72676449  & 2.28187158  & 2.28187158   \\
    & & 396290 &-1.72622170  & -1.72622170  & 2.28197823  & 2.28197823\\
    \cmidrule{3-7} 
    &\multirow{4}{*}{FEM}  & 4130 &-1.73749092  & -1.73749551  & 2.27973482  & 2.27973485  \\
                               & & 16514  &-1.72892024  & -1.72892100  & 2.28144492  & 2.28144492  \\
                               & & 66050  &-1.72676323  & -1.72676333  & 2.28187160  & 2.28187160  \\
                               & & 264194 &-1.72622154  & -1.72622155  & 2.28197823  & 2.28197823  \\
\bottomrule
\end{tabular*}\end{minipage}
\end{table}

\noindent {\bf Example 4:} In this example, we consider the case where $n$ is a complex number.
Let $\Omega=B$ be a disk centered at the origin with a radius of 1.5 (see Fig.~\ref{fig:11}),
and let $D_1$ be a square with vertices at the coordinates and a side length of $\sqrt{2}$. The
region $D_2$ is an L-shaped domain given by $[-0.9, 1.1] \times [-1.1, 0.9] \backslash (0.1, 1.1) \times (-1.1, -0.1)$, while $D_3$ is taken
as $\Omega_C$. We set the parameters on $D_\kappa~(\kappa=1, 2, 3)$ according to Case 2, and set $n=1$ on $\Omega \backslash D_\kappa$ ($\kappa=1, 2, 3$). Using MATLAB's pdetool, the mesh generation we employed is shown in Fig.~\ref{fig:11}.
We compute the first five approximate eigenvalues of the problem using the VEM with degrees \(l=1\) and \(l=2\), and the results are listed in Tables \ref{tab10}-\ref{tab12}.
When $l=1$, we refer to the exact modified transmission eigenvalues given in Example 7-9 of \cite{Liu2023}.
The comparison between Tables \ref{tab10}-\ref{tab12} and the results in \cite{Liu2023} indicates that our method is effective in
computing the modified transmission eigenvalues with absorbing media.

\begin{figure}[ht]
\captionsetup[subfloat]{labelformat=empty} 
    \centering
    \subfloat[$\widehat{\mathcal{T}_h^1}$]{\label{fig1a}\includegraphics[width=.33\linewidth]{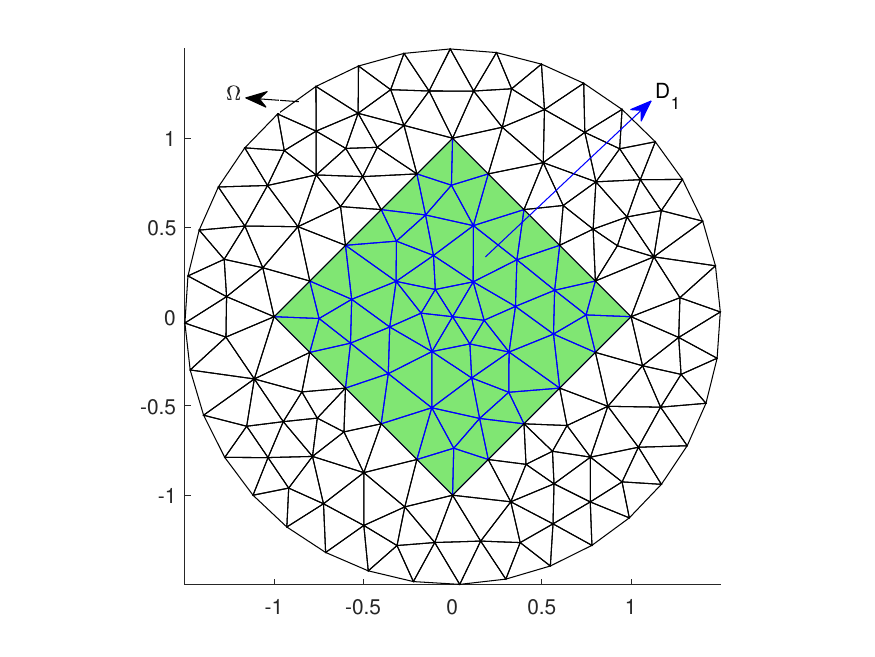}}
    \subfloat[$\widehat{\mathcal{T}_h^2}$]{\label{fig1a}\includegraphics[width=.33\linewidth]{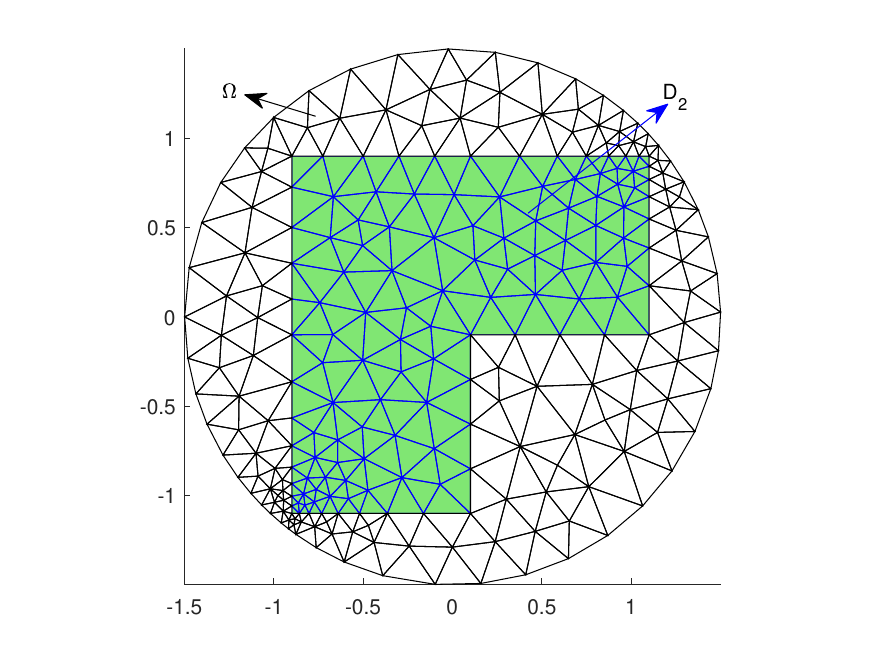}}
    \subfloat[$\widehat{\mathcal{T}_h^3}$]{\label{fig1b}\includegraphics[width=.33\linewidth]{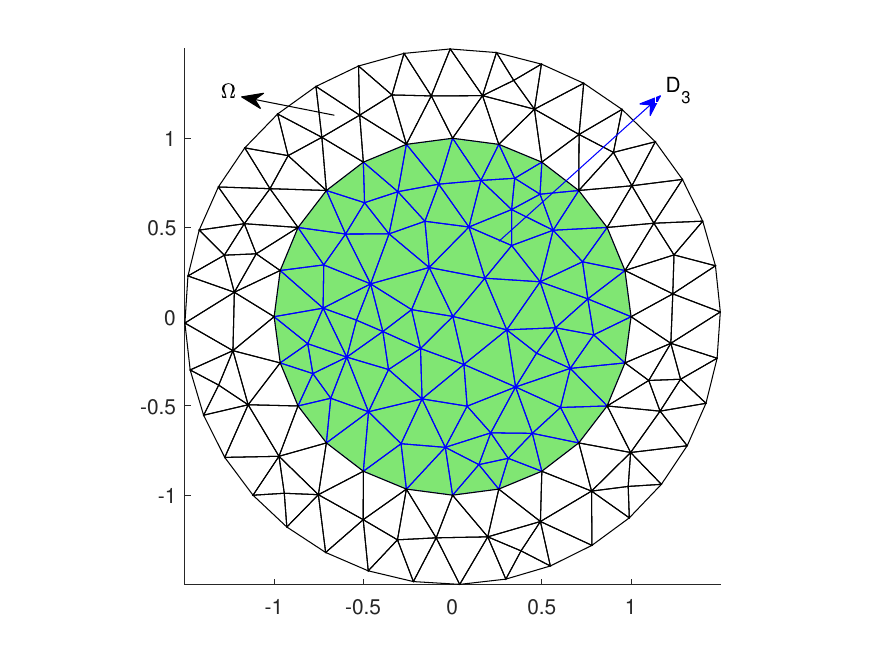}}
    \caption{The generated mesh of Example 4}
    \label{fig:11}
\end{figure}

\begin{table}[ht]
\setlength{\tabcolsep}{1pt} 
\centering\begin{minipage}{\textwidth}
\caption{The first five approximate eigenvalues obtained by solving with VEM on the generated mesh $\widehat{\mathcal{T}_h^1}$}
\label{tab10}
\begin{tabular*}{\textwidth}{@{\extracolsep{\fill}}*{6}{l}@{\extracolsep{\fill}}}
\toprule%
$l$  &$Ndof$ &$\lambda_{1,h}=\lambda_{2,h}$(Rate)  &$\lambda_{3,h}$(Rate) &$\lambda_{4,h}$(Rate) &$\lambda_{5,h}$(Rate) \\
\midrule
 \multirow[c]{8}{*}{$1$}
    &\multirow[t]{2}{*}{266}   &0.1889+0.5283$\mathrm{i}$    & 0.5418+1.0421$\mathrm{i}$  & -1.8446+0.2723$\mathrm{i}$  &-1.8633+0.1121$\mathrm{i}$   \\
                   &           &(-) &(-) &(-) &(-)\\
    &\multirow[t]{2}{*}{1058}  &0.1876+0.5257$\mathrm{i}$    & 0.5232+1.0269$\mathrm{i}$  & -1.8619+0.1130$\mathrm{i}$  & -1.8498+0.2754$\mathrm{i}$ \\
                   &           &(0.989)     &(1.029)     &(4.427)     &(3.043) \\
    &\multirow[t]{2}{*}{4226}  &0.1872+0.5250$\mathrm{i}$    & 0.5185+1.0229$\mathrm{i}$  & -1.8619+0.1133$\mathrm{i}$  & -1.8515+0.2763$\mathrm{i}$  \\
                   &           &(1.146)     &(1.158 )     &(1.128)     &(1.123) \\
    &\multirow[t]{2}{*}{16898} &0.1871+0.5248$\mathrm{i}$    & 0.5173+1.0219$\mathrm{i}$  & -1.8619+0.1134$\mathrm{i}$   & -1.8519+0.2766$\mathrm{i}$\\
                   &           &(-)     &(-)     &(-)     &(-) \\
                                 \cmidrule{2-6}
  \multirow[c]{6}{*}{$2$}
    &\multirow[t]{2}{*}{1586}  &0.1827+0.5297$\mathrm{i}$   & 0.5208+1.0299$\mathrm{i}$  & -1.8783+0.1149$\mathrm{i}$  & -1.8682+0.2805$\mathrm{i}$   \\
    &           &(-)     &(-)     &(-)     &(-) \\
    &\multirow[t]{2}{*}{6338}  &0.1859+0.5260$\mathrm{i}$   & 0.5179+1.0236$\mathrm{i}$  & -1.8660+0.1138$\mathrm{i}$  & -1.8561+0.2776$\mathrm{i}$   \\
    &           &(1.160)     &(1.168)     &(1.156)     &(1.156) \\
    &\multirow[t]{2}{*}{25346} &0.1867+0.5251$\mathrm{i}$   & 0.5172+1.0221$\mathrm{i}$  & -1.8629+0.1135$\mathrm{i}$  & -1.8531+0.2769$\mathrm{i}$ \\ 
    &           &(-)     &(-)     &(-)     &(-) \\
\bottomrule
\end{tabular*}\end{minipage}
\end{table}

\begin{table}[ht]
\setlength{\tabcolsep}{1pt} 
\centering\begin{minipage}{\textwidth}
\caption{The first five approximate eigenvalues obtained by solving with VEM on the generated mesh of $\widehat{\mathcal{T}_h^2}$}
\label{tab11}
\begin{tabular*}{\textwidth}{@{\extracolsep{\fill}}*{7}{l}@{\extracolsep{\fill}}}
\toprule%
\multicolumn{1}{c}{$l$}  &$Ndof$ &$\lambda_{1,h}$ &$\lambda_{2,h}$  &$\lambda_{3,h}$ &$\lambda_{4,h}$ &$\lambda_{5,h}$       \\
\midrule
 \multirow{8}{*}{$1$}
 &\multirow[t]{2}{*}{454}  &0.1055+0.8085$\mathrm{i}$ &0.2193+1.8710$\mathrm{i}$ &-1.8038+0.6405$\mathrm{i}$ &-1.5246+1.2254$\mathrm{i}$ &1.8679+1.3493$\mathrm{i}$\\
    &           &(-)     &(-)     &(-)     &(-)   &(-)\\
 &\multirow[t]{2}{*}{1810} &0.1052+0.8013$\mathrm{i}$ &0.2116+1.8661$\mathrm{i}$ &-1.8040+0.6488$\mathrm{i}$ &-1.5289+1.2412$\mathrm{i}$ &1.9092+1.3175$\mathrm{i}$\\
    &           &(1.015)     &(1.018)     &(0.968)     &(0.987)   &(1.020)\\
&\multirow[t]{2}{*}{7234}  &0.1051+0.7995$\mathrm{i}$ &0.2097+1.8648$\mathrm{i}$ &-1.8043+0.6511$\mathrm{i}$ &-1.5301+1.2455$\mathrm{i}$ &1.9188+1.3081$\mathrm{i}$\\
    &           &(1.155)     &(1.156)     &(1.141)     &(1.147)   &(1.157)\\
&\multirow[t]{2}{*}{28930} &0.1050+0.7990$\mathrm{i}$ &0.2092+1.8645$\mathrm{i}$ &-1.8044+0.6517$\mathrm{i}$ &-1.5304+1.2466$\mathrm{i}$ &1.9211+1.3057$\mathrm{i}$\\
    &           &(-)     &(-)     &(-)     &(-)   &(-)\\
                                 \cmidrule{2-7}
  \multirow{6}{*}{$2$}
&\multirow[t]{2}{*}{2714} &0.0992+0.8063$\mathrm{i}$  &0.2115+1.8745$\mathrm{i}$ &-1.8219+0.6594$\mathrm{i}$ &-1.5344+1.2538$\mathrm{i}$ &1.9196+1.3087$\mathrm{i}$\\
   &           &(-)     &(-)     &(-)     &(-)   &(-)\\
&\multirow[t]{2}{*}{10850} &0.1036+0.8007$\mathrm{i}$ &0.2096+1.8669$\mathrm{i}$ &-1.8088+0.6538$\mathrm{i}$ &-1.5315+1.2487$\mathrm{i}$ &1.9213+1.3058$\mathrm{i}$\\
 &           &(1.162)     &(1.159)     &(1.158)     &(1.145)   &(1.149)\\
&\multirow[t]{2}{*}{43394} &0.1047+0.7993$\mathrm{i}$ &0.2092+1.8650$\mathrm{i}$ &-1.8055+0.6523$\mathrm{i}$ &-1.5308+1.2474$\mathrm{i}$ &1.9217+1.3051$\mathrm{i}$\\
&           &(-)     &(-)     &(-)     &(-)   &(-)\\
\bottomrule
\end{tabular*}\end{minipage}
\end{table}

\begin{table}[ht]
\centering\begin{minipage}{\textwidth}
\caption{The first five approximate eigenvalues obtained by solving with VEM on the generated mesh of $\widehat{\mathcal{T}_h^3}$}
\label{tab12}
\begin{tabular*}{\textwidth}{@{\extracolsep{\fill}}*{5}{l}@{\extracolsep{\fill}}}
\toprule%
$l$  &$Ndof$ &$\lambda_{1,h}=\lambda_{2,h}(Rate)$  &$\lambda_{3,h}=\lambda_{4,h}(Rate)$ &$\lambda_{5,h}(Rate)$       \\
\midrule
 \multirow{4}{*}{$1$}
    & 256   &0.1303+1.1681$\mathrm{i}$(-)        & -1.8352+0.6049$\mathrm{i}$(-)        & 0.5188+2.0035$\mathrm{i}$(-)   \\
    & 1018  &0.1230+1.1796$\mathrm{i}$(1.021)    & -1.8469+0.6295$\mathrm{i}$(0.990)    & 0.4817+1.9743$\mathrm{i}$(1.049)   \\
    & 4066  &0.1210+1.1825$\mathrm{i}$(1.156)    & -1.8504+0.6361$\mathrm{i}$(1.148)    & 0.4727+1.9670$\mathrm{i}$(1.162)   \\
    & 16258 &0.1204+1.1832$\mathrm{i}$(-)        & -1.8513+0.6377$\mathrm{i}$(-)        & 0.4704+1.9651$\mathrm{i}$(-)   \\
                                 \cmidrule{2-5}
  \multirow{3}{*}{$2$}
    & 1526  &0.1179+1.1693$\mathrm{i}$(-)    & -1.8674+0.6262$\mathrm{i}$(-)    & 0.4800+1.9473$\mathrm{i}$(-)   \\
    & 6098  &0.1197+1.1799$\mathrm{i}$(1.161)       & -1.8556+0.6353$\mathrm{i}$(1.158)       & 0.4723+1.9603$\mathrm{i}$(1.161)   \\
    & 24386 &0.1201+1.1825$\mathrm{i}$(-)    & -1.8526+0.6375$\mathrm{i}$(-)    & 0.4704+1.9635$\mathrm{i}$(-)   \\
\bottomrule
\end{tabular*}\end{minipage}
\end{table}

\noindent {\bf Example 5:} In this example, we consider the case where $n$ is a variable coefficient. First, we compute the first approximate eigenpair of \eqref{a2.3} on a square of side length $2$, where the coefficients in the first and third quadrants are set as in Case 3, and those in the second and fourth quadrants are set as in Case 1. Furthermore, under the same coefficient setting as in Case 3, we compute the first four approximate eigenvalues of \eqref{a2.3} on a domain with a hole obtained by removing a unit disk from a square of side length $2$. The corresponding initial mesh is shown in Fig.~\ref{fig:12}. Since the exact eigenpairs are unknown, we adopt the same reference value selection strategy as that used in Example 1, and when plotting the eigenfunction error curves, we use the $L^2$ norm of the difference between the numerical eigenfunctions on adjacent meshes, $\|u_{1,h/2} - u_{1,h}\|_0$, as a substitute for the error $\|u_1 - u_{1,h}\|_0$, as shown in Fig.~\ref{fig:13}.\\
\indent From the left panel of Fig.~\ref{fig:13}, we observe that the error curves for the eigenvalues and eigenfunctions are generally parallel to the lines with slopes $-1$ and $-0.5$, respectively, indicating that our method achieves the optimal convergence order for the setting with discontinuous variable coefficients. Similarly, the right panel of Fig.~\ref{fig:13} demonstrates that our method remains effective for complex domains with holes.
\begin{figure}[ht]
\captionsetup[subfloat]{labelformat=empty} 
    \centering
    \subfloat[$\mathcal{T}_h^4$]{\label{fig1a}\includegraphics[width=.45\linewidth]{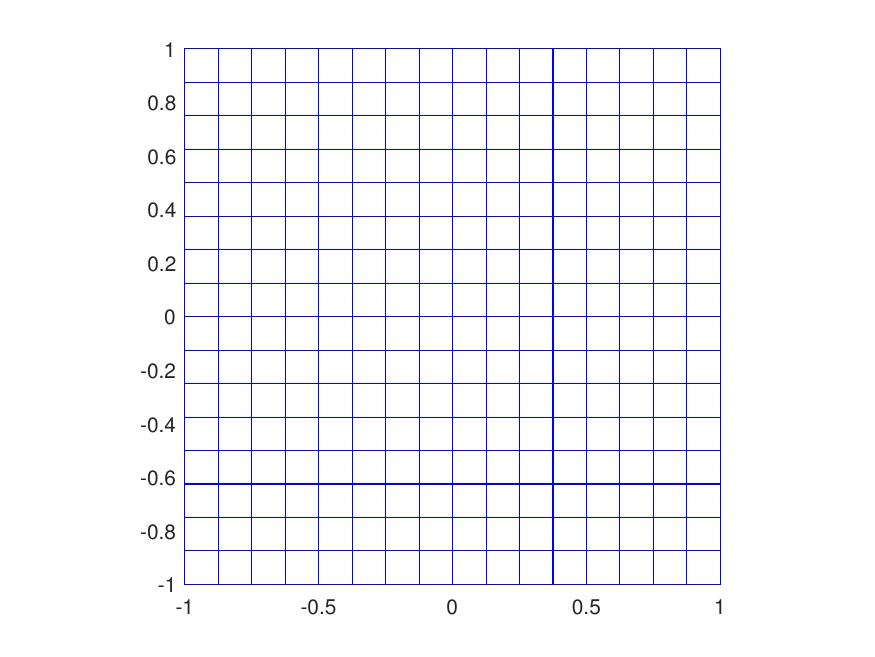}}
    \subfloat[$\mathcal{T}_h^5$]{\label{fig1b}\includegraphics[width=.45\linewidth]{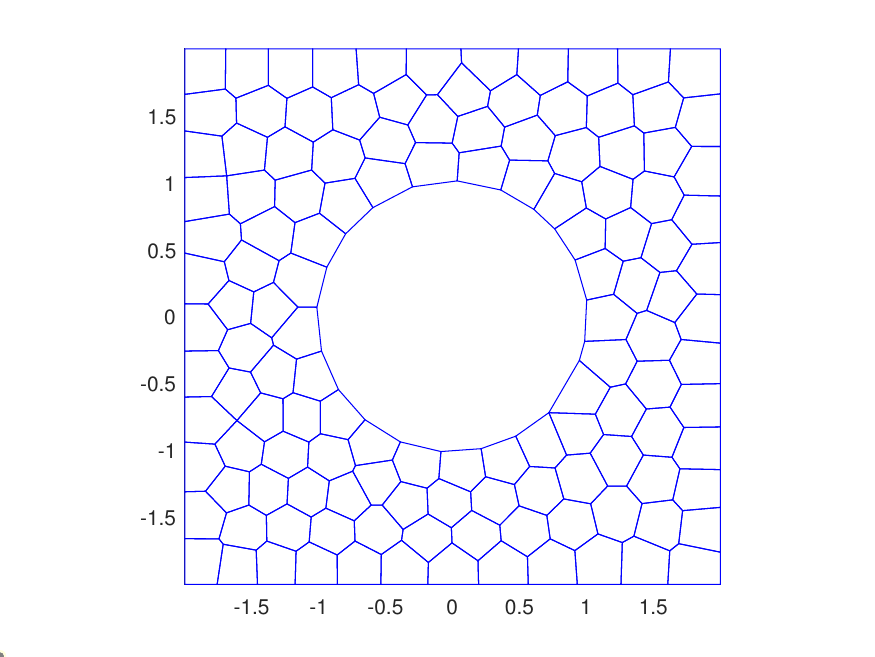}}
    \caption{The initial mesh of Example 5}
    \label{fig:12}
\end{figure}

\begin{figure}[ht]
\captionsetup[subfloat]{labelformat=empty} 
    \centering
    \subfloat[Study of convergence when $ l = 1 $]{\label{fig1a}\includegraphics[width=.45\linewidth]{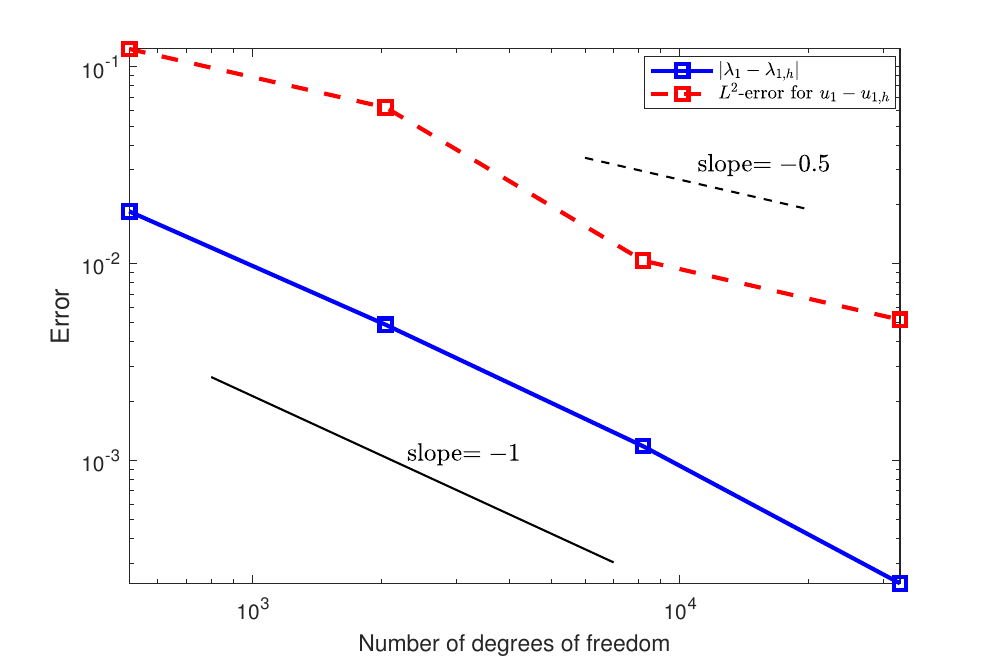}}
    \subfloat[Study of convergence when $ l = 1 $]{\label{fig1b}\includegraphics[width=.45\linewidth]{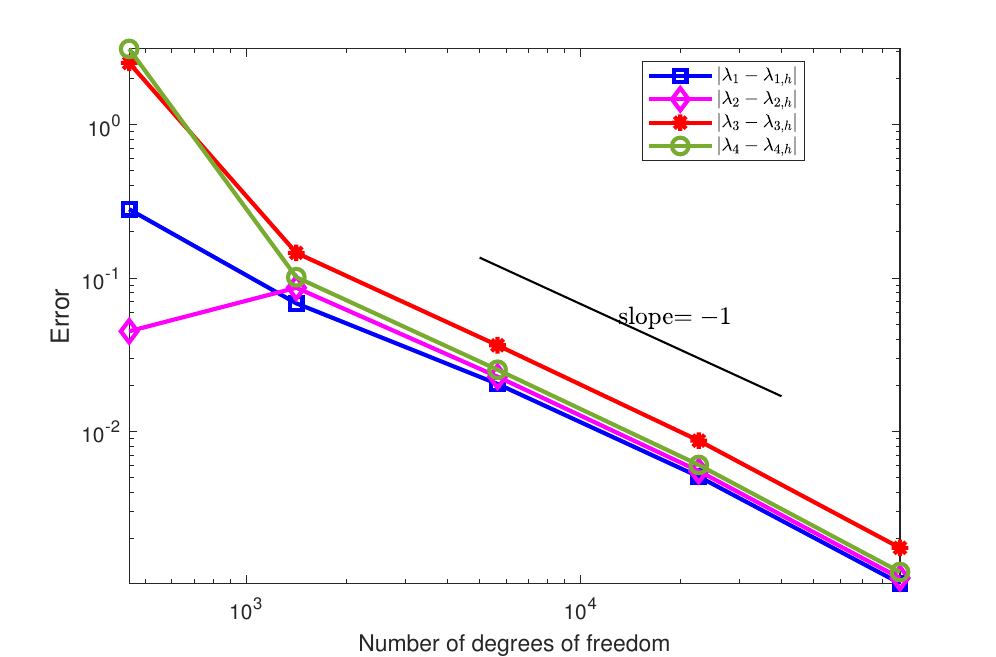}}
    \caption{Example 5 error curves: first eigenvalue/eigenfunction errors on $\mathcal{T}_h^4$ (left) and first four eigenvalues errors on $\mathcal{T}_h^5$ (right)}
    \label{fig:13}
\end{figure}

\section*{Acknowledgment}The author would like to express their sincere gratitude to the anonymous reviewers for their careful reading of the manuscript, as well as their coments that led to a considerable improvement of the original manuscript. Special thanks go to Professor Yidu Yang for his diligent guidance and support throughout the research and writing process.

\section*{Conflict of interest}
The authors declare no potential conflict of interests.\\

\end{document}